\documentclass[preprint,5pt]{elsarticle}
\usepackage{multicol}
\usepackage{color}
\usepackage{xspace}
\usepackage{enumerate}
\usepackage{amssymb}
\usepackage{graphicx}
\usepackage{amsthm}
\usepackage{amsfonts}
\usepackage{amsmath}
\usepackage{MnSymbol}

\makeatletter
\def\bs{\expandafter\@gobble\string\\}
\def\lb{\expandafter\@gobble\string\{}
\def\rb{\expandafter\@gobble\string\}}
\def\@pdfauthor{C.V.Radhakrishnan}
\def\@pdftitle{elsarticle.cls -- A documentation}
\def\@pdfsubject{Document formatting with elsarticle.cls}
\def\@pdfkeywords{LaTeX, Elsevier Ltd, document class}

\DeclareRobustCommand{\LaTeX}{L\kern-.26em%
        {\sbox\z@ T%
         \vbox to\ht\z@{\hbox{\check@mathfonts
           \fontsize\sf@size\z@
           \math@fontsfalse\selectfont
          A\,}%
         \vss}%
        }%
     \kern-.15em%
    \TeX}
\makeatother

\newtheorem{theorem}{Theorem}

\newtheorem{proposition}[theorem]{Proposition}
\newtheorem{definition}[theorem]{Definition}
\newdefinition{rmk}{Remark}
\newdefinition{example}{Example}

\newcommand{\vertiii}[1]{{\left\vert\kern-0.25ex\left\vert\kern-0.25ex\left\vert #1 
    \right\vert\kern-0.25ex\right\vert\kern-0.25ex\right\vert}}

\begin{document}

\def\testa{This is a specimen document. }
\def\testc{\testa\testa\testa\testa}
\def\testb{\testc\testc\testc\testc\testc}
\long\def\test{\testb\par\testb\par\testb\par}


\title{A time--dependent FEM-BEM coupling method for fluid--structure interaction {in $3d$}}

\author[1,2]{Heiko Gimperlein\corref{cor1}}\ead{h.gimperlein@hw.ac.uk}
\author[3,4]{Ceyhun \"{O}zdemir\fnref{fn2}}\ead{oezdemir@ifam.uni-hannover.de}
\author[3]{Ernst P. Stephan} \ead{stephan@ifam.uni-hannover.de}
\cortext[cor1]{Corresponding author}
\fntext[fn2]{C.~\"{O}zdemir acknowledges support by the Avicenna foundation.}
\address[1]{Maxwell Institute for Mathematical Sciences and Department of Mathematics, Heriot--Watt University, Edinburgh, EH14 4AS, United Kingdom}
\address[2]{Institute for Mathematics, University of Paderborn, Warburger Str.~100, 33098 Paderborn, Germany}
\address[3]{Institute of Applied Mathematics, Leibniz University Hannover, 30167 Hannover, Germany}
\address[4]{Institute for Mechanics,  Graz University of Technology, 8010 Graz, Austria}
\date{\today}
\begin{abstract}
We consider the well-posedness and a priori error estimates of a $3d$ FEM-BEM coupling method for fluid-structure interaction in the time domain. For an elastic body immersed in a fluid, the exterior linear wave equation for the fluid is reduced to an integral equation on the boundary involving the Poincar\'{e}-Steklov operator. The resulting problem is solved using a Galerkin boundary element method in the time domain, coupled to a finite element method for the Lam\'{e} equation inside the elastic body. Based on ideas from the time--independent coupling formulation, we obtain an a priori  error estimate  and discuss the implementation of the proposed method. Numerical experiments illustrate the performance of our scheme for model problems.
\end{abstract}
\begin{keyword}
Fluid-structure interaction; FEM-BEM coupling; space-time methods; a priori error estimate; wave equation. 
\end{keyword}

\maketitle

\section{Introduction}

Coupled finite and boundary element procedures provide an efficient
and extensively investigated tool for the numerical solution of elliptic interface
and contact problems, particularly in unbounded domains \cite{gwinsteph, ency}. On the other
hand, much of the current interest in boundary element procedures focuses
on hyperbolic problems in the time domain, both on Galerkin methods and
convolution quadrature {\cite{aimi, costabel04,gimperleinreview, jr, MR3468871}.}\\

To compute the scattering of time-dependent waves by a bounded, penetrable obstacle, the coupling of time domain finite elements (FEM) and boundary elements (BEM) becomes relevant. The recent mathematical analysis of FEM-BEM coupling in the time domain was initiated in \cite{ajrt}, coupling discontinuous finite elements to Galerkin boundary elements for the $3d$ wave equation. A general analysis of the coupling between different discretizations for acoustic wave equations was provided in \cite{bls}, with a focus on convolution quadrature. Since then, FEM-BEM coupling for convolution quadrature methods has been applied in a variety of applications {in $2$ dimensions}, such as  fluid-structure and  fluid-thermoelastic problems, as well as nonlinear elastic problems involving piezoelectric scatterers \cite{MR3614885, MR3596550, thermo, svs}. For time-dependent Galerkin methods {and their application to $3d$ problems}, on the other hand, much less is known. In addition to \cite{ajrt}, previous related work  includes the  energy-based formulations  investigated by Aimi and collaborators for wave-wave coupling {in $3$d multidomains and layered media} \cite{aimi1, aimi2}. \\

In this article we study a simple space-time Galerkin FEM-BEM coupling method for  fluid-structure interaction, describing  the transient scattering of waves in an inviscid homogeneous fluid by an elastic obstacle. Based on ideas from time--independent coupling formulations  \cite{bmz, dsm} and the analysis in the frequency domain \cite{MR3614885}, we present a basic a priori  error estimate for a space-time Galerkin approximation in anisotropic Sobolev spaces \cite{BamHa}. {We discuss in detail the numerical implemention of our proposed coupling method in $3d$. Numerical experiments for model problems illustrate the performance of the scheme.} \\

To describe the results of this article in more detail, recall the equations for an elastic body submersed in a fluid. Let $\Omega \subset \mathbb{R}^3$ be a bounded Lipschitz domain, and $\Omega^c = \mathbb{R}^3 \backslash \overline{\Omega}$. The elastic deformation $\mathbf{u}$ in $\Omega$ is described by the Lam\'{e} operator $\Delta^{*} {\mathbf{u}}= \mu \Delta {\mathbf{u}} + (\lambda+\mu) \nabla(\mathrm{div}{\mathbf{u}})$, with Lam\'e constants $\mu \geq 0 $ and $\lambda$, such that $3 \lambda + 2 \mu \geq 0$. 
The deformation is coupled to the wave equation in $\Omega^c$, leading to the coupled interface problem
\begin{subequations}\label{eq:wholeproblematonce}
\textcolor{black}{\begin{equation}\label{eq:waveeqJ}
 \ddot{v} - c^2\Delta v = 0, \quad (x,t) \in \Omega^{c} \times (0,\infty),
\end{equation}
\begin{equation}\label{eq:lameeqJ}
\rho_{1} \ddot{\mathbf{u}} - \Delta^{*} {\mathbf{u}} = 0, \quad (x,t) \in \Omega \times (0,\infty),
\end{equation}    
\begin{equation}\label{eq:tranmissioncond1_vectorfield_J}
\rho_{2} \tilde{\sigma}({\mathbf{u}}) \cdot n + \dot{v} \cdot n = - \dot{v}^{inc} \cdot n \quad \text{ on } \Gamma \times (0,\infty),
\end{equation}
\begin{equation}\label{eq:tranmissioncond2_scalarfield_J}
\dot{\mathbf{u}} \cdot n + {\partial_n^{+} v } = - \partial_n^{+} v^{inc} \quad \text{ on } \Gamma \times (0,\infty), 
\end{equation}
\begin{equation}\label{eq:exteriorinitialcondJ}
v(x,0)=\dot{v}(x,0)=0 \quad \text{ in } \Omega^{c}, 
\end{equation}
\begin{equation}\label{eq:interiorinitialcondJ}
{\mathbf{u}}(x,0)=\dot{\mathbf{u}}(x,0)=0 \quad \text{ in } \Omega \ ,
\end{equation}}
\end{subequations}
\textcolor{black}{for a given incident wave $v^{inc}$ in $\Omega^{c}$.} 
Here, the stress is given in terms of the deformation  as $\tilde{\sigma}({\mathbf{u}}) = (\lambda \mathrm{div} {\mathbf{u}}) I + 2 \mu \varepsilon({\mathbf{u}})$, $\varepsilon({\mathbf{u}})=\frac{1}{2}((\nabla {\mathbf{u}}) +(\nabla {\mathbf{u}})^{T})$, with $I$ the identity matrix. \textcolor{black}{Time derivatives are denoted by a dot, and} $n$ is the outward-pointing unit normal vector to $\partial \Omega$. \textcolor{black}{The Neumann trace on $\Gamma$ from the exterior domain $\Omega^c$ is denoted by $\partial_{n}^{+}$, while $\partial_n^-$ is the corresponding Neumann trace from the  interior domain $\Omega$}. We choose units in which $c=\rho_1 = \rho_{2}=1$. \\

To solve this interface problem numerically,  we use the Poincar\'{e}-Steklov operator for the exterior wave equation to reformulate  it as a coupled domain / boundary integral equation in $\Omega$ and $\Gamma$. The \textcolor{black}{Poincar\'{e}}-Steklov operator is expressed in terms of layer potentials for the wave equation, as known for time-independent symmetric FEM-BEM coupling methods. The resulting space-time weak formulation is approximated using finite elements in $\Omega$ and Galerkin
boundary elements on $\Gamma$, based on tensor products of piecewise polynomial functions on a quasi-uniform
mesh in space and a uniform mesh in time. Our a priori estimates assure convergence. We discuss a numerical implementation in detail and study the numerical performance of the method.  \\

\noindent \emph{The article is organized as follows:} Section \ref{sec2} reformulates the coupled problem \eqref{eq:wholeproblematonce} as a domain / boundary integral equation in $\Omega$ and $\Gamma$ and discusses its discretization and well-posedness. It provides the basis for the derivation of an a priori error estimate in Section \ref{sec3}. Numerical examples and related algorithmic considerations are the content of  Section \ref{sec4}. Two appendices discuss boundary integral operators for the wave equation and the appropriate space-time anisotropic Sobolev spaces, as well as detailed algorithmic aspects of the proposed scheme. \\

\noindent \emph{Notation:} To simplify notation, we will write $f \lesssim g$, if there exists a constant $C>0$ independent of the arguments of the functions $f$ and $g$ such that $f \leq C g$. We will write $f \lesssim_\sigma g$, if $C$ may depend \textcolor{black}{on a parameter} $\sigma$.
\textcolor{black}{For a function $v$ on $\textbf{R}^n\setminus \Gamma$, we denote by $v^{\pm}|_{\Gamma}$ the trace of $v$ on $\Gamma$ from the exterior domain $\Omega^{c}$, resp.~from the interior domain $\Omega$}.

\section{Weak formulation and FEM-BEM coupling}\label{sec2}

Recall that the fluid-structure interaction problem  \eqref{eq:wholeproblematonce} is well-posed \cite{Filipe}:
\begin{theorem}\label{wellposed}
Let $\sigma >0$, $s \geq 0$ and assume that $v^{inc}|_\Gamma \in H_\sigma^{3+s}(\mathbb{R}^+, {H}^{\frac{1}{2}}(\Gamma))$, $\textcolor{black}{\partial_n^{+} v^{inc}} \in H_\sigma^{3+s}(\mathbb{R}^+, {H}^{-\frac{1}{2}}(\Gamma))$. Then the system \eqref{eq:wholeproblematonce} admits a unique solution $(\mathbf{u},v) \in H_\sigma^{1+s}(\mathbb{R}^+, {H}^{1}(\Omega)) \times  H_\sigma^{s}(\mathbb{R}^+, {H}^{1}(\Omega^c))$, which depends continuously on the data.
\end{theorem}
\textcolor{black}{See Appendix A for the definitions of the relevant space-time anisotropic Sobolev spaces, depending on a weight parameter $\sigma>0$.}\\

In this section we reformulate problem \eqref{eq:wholeproblematonce} as a coupled domain / boundary integral equation and propose a finite element / boundary element coupling method for its numerical solution.

 The key ingredient to reduce equation \eqref{eq:waveeqJ} in $\Omega^c$ to $\Gamma$ is the retarded Poincar\'{e}-Steklov operator $\mathcal{S}$ on $\Gamma$, defined as 
$\mathcal{S} v^{\textcolor{black}{+}}|_\Gamma = \textcolor{black}{{\partial_{n}^{+} v}}$ for a solution $v$ of \eqref{eq:waveeqJ}.

Using $\mathcal{S}$, equation \eqref{eq:tranmissioncond2_scalarfield_J}  becomes:

\textcolor{black}{\begin{equation} \label{transmS}
	\textstyle -\dot{\mathbf{u}} \cdot n - \mathcal{S} v^{\textcolor{black}{+}}|_\Gamma = {\partial_n^{+} v^{inc}} \ \ \ \text{ on } \Gamma.
\end{equation}}

To compute the Poincar\'{e}-Steklov operator $\mathcal{S}$, we use the following formulation in terms of  boundary integral operators \cite{gwinsteph}, see Appendix A:
\begin{equation*}
\mathcal{S} v^{\textcolor{black}{+}}|_\Gamma = W v^{\textcolor{black}{+}}|_\Gamma \textcolor{black}{-} \textstyle{(K^{T}-\frac{1}{2})V^{-1}(K-\frac{1}{2} )} v^{\textcolor{black}{+}}|_\Gamma \ .
 \end{equation*}
In terms of $\phi=v^{\textcolor{black}{+}}|_\Gamma$ and an auxiliary variable $\lambda = V^{-1}(K-\frac{1}{2})\phi$, \eqref{transmS} becomes 


\textcolor{black}{\begin{equation*}
 \textstyle{- \dot{\mathbf{u}} \cdot n } - W \phi + \textstyle{(K^{T}-\frac{1}{2})} \lambda = \partial_{n}^{+} v^{inc} \ . \end{equation*}}

Problem \eqref{eq:wholeproblematonce} is therefore equivalent to the following system on $\Omega$ and $\Gamma$:
\begin{subequations}\label{allinbio}
\textcolor{black}{
\begin{equation}\label{eq:lameeq2J}
\textstyle \ddot{\mathbf{u}} - \Delta^{*} {\mathbf{u}} = 0, \quad (x,t) \in \Omega \times (0,\infty),
\end{equation}    
\begin{equation}\label{eq:tranmissioncond1_vectorfield2_J}
\textstyle \tilde{\sigma}(\mathbf{u}) \cdot n + \dot{\phi} \cdot n = - \dot{v}^{inc} \cdot n \quad \text{ on } \Gamma \times (0,\infty),  
\end{equation}
\begin{equation}\label{eq:tranmissioncond2_scalarfield2_J}
\textstyle - \dot{\mathbf{u}} \cdot n - W \phi + {\textstyle (K^{T}-\frac{1}{2})} \lambda =  {\partial_{n}^{+}}v^{inc} \quad \text{ on } \Gamma \times (0,\infty) ,
\end{equation}
\begin{equation}\label{eq:representationeq_J}
 {\textstyle (\frac{1}{2} - K)} \phi + V \lambda = 0  \quad \text{ on } \Gamma \times (0,\infty),
\end{equation}
\begin{equation}\label{eq:exteriorinitialrepresentcondJ}
\phi(x,0)=\dot{\phi}(x,0)=\lambda(x,0)=0  \quad \text{ on } \Gamma ,
\end{equation}
\begin{equation}\label{eq:interiorinitialcondJ}
{\mathbf{u}}(x,0)=\dot{\mathbf{u}}(x,0)=0 \quad \text{ in } \Omega.
\end{equation}
}
\end{subequations}
The solution $v$ in $\Omega^c$ is recovered from the representation formula $v = D\phi - S \lambda$, using the single and double layer potentials from Appendix A.\\

\textcolor{black}{We derive a weak formulation of \eqref{allinbio} in the weighted $L^2$-Sobolev spaces from Theorem \ref{wellposed}. Recall the weighted $L^2$ inner products for given $\sigma>0$:
\begin{equation*}
 (\mathbf{u},\mathbf{v})_{\Omega\times\mathbb{R}^{+},\sigma} := \int_{0}^{\infty} e^{-2 \sigma t} \int_{\Omega} \mathbf{u} \cdot \mathbf{v} dx dt \text{ and } \langle\mathbf{u},\mathbf{v}\rangle_{\Gamma\times\mathbb{R}^{+},\sigma} := \int_{0}^{\infty} e^{-2 \sigma t} \int_{\Gamma} \mathbf{u} \cdot \mathbf{v} ds_x dt .
\end{equation*}}
Given a smooth solution $(\mathbf{u}, \phi, \lambda)$ of \eqref{allinbio}, \textcolor{black}{equations \eqref{eq:lameeq2J} and \eqref{eq:tranmissioncond1_vectorfield2_J} combined} with Betti's formula,
\begin{equation*}
  \langle \tilde{\sigma}(\mathbf{u}) \cdot n, \mathbf{w}|_\Gamma \rangle_{\Gamma \times \mathbb{R}^{+},\sigma} = (\tilde{\sigma}({\mathbf{u}}),\varepsilon({\mathbf{w}}))_{\Omega\times \mathbb{R}^{+},\sigma} + (\Delta^{*} {\mathbf{u}} , {\mathbf{w}})_{\Omega\times \mathbb{R}^{+},\sigma}\ , 
\end{equation*}
lead to
\begin{equation}\label{eq:fsi_interior}
  (\tilde{\sigma}({\mathbf{u}}),\varepsilon(\dot{\mathbf{w}}))_{\Omega\times \mathbb{R}^{+},\sigma} + (\ddot{{\mathbf{u}}}, \dot{\mathbf{w}})_{\Omega \times \mathbb{R}^{+},\sigma}+ \langle \dot{\phi} ,  \dot{\mathbf{w}}|_\Gamma \cdot n \rangle_{\Gamma\times \mathbb{R}^{+},\sigma} = - \langle \dot{v}^{inc} ,  \dot{\mathbf{w}}|_\Gamma \cdot n \rangle_{\Gamma\times \mathbb{R}^{+},\sigma} \ .
\end{equation}
The first two terms on the left hand side define a bilinear form
\begin{align*}
 a({\mathbf{u}},\dot{\mathbf{w}}) := (\tilde{\sigma}({\mathbf{u}}),\varepsilon(\dot{\mathbf{w}}))_{\Omega \times \mathbb{R}^{+},\sigma} + (\ddot{{\mathbf{u}}}, \dot{\mathbf{w}})_{\Omega \times \mathbb{R}^{+},\sigma}\ .
\end{align*}


Adding \eqref{eq:fsi_interior} and the weak formulations of \eqref{eq:tranmissioncond2_scalarfield2_J} and  \eqref{eq:representationeq_J}, one obtains  a weak formulation of \eqref{allinbio} in $\textcolor{black}{\widetilde{X}} := H_{\sigma}^{1}(\mathbb{R}^{+},H^{1}(\Omega))^{3} \times  H_{\sigma}^{1}(\mathbb{R}^{+},H^{1/2}(\Gamma)) \times H_{\sigma}^{1}(\mathbb{R}^{+},H^{-1/2}(\Gamma))$:\\ 

\noindent Find $(\mathbf{u},\phi,\lambda) \in \widetilde{X}$ such that for all $({\mathbf{w}},w,m) \in \widetilde{X}$
\begin{flalign}
 &a(\mathbf{u},\dot{\mathbf{w}}) + \langle \dot{\phi}, \dot{\mathbf{w}}|_\Gamma \cdot n \rangle_{\Gamma\times \mathbb{R}^{+}\!\!,\sigma} \!-\! \langle \dot{{\mathbf{u}}}|_\Gamma \cdot n , \dot{w} \rangle_{\Gamma\times \mathbb{R}^{+}\!\!,\sigma} \!-\! \langle W \phi , \dot{w} \rangle_{\Gamma\times \mathbb{R}^{+}\!\!,\sigma} \nonumber
 \\ &\qquad - \textstyle \langle (\frac{1}{2} I - K^{T}) \lambda , \dot{w} \rangle_{\Gamma\times \mathbb{R}^{+}\!\!,\sigma} + \!\textstyle \langle (\frac{1}{2} I - K) \phi , \dot{m} \rangle_{\Gamma\times \mathbb{R}^{+}\!\!,\sigma} \!+\! \langle V \lambda , \dot{m} \rangle_{\!\Gamma\times \mathbb{R}^{+}\!\!,\sigma} \label{eq:FSIsolkap4}\\&\qquad \qquad =\! - \langle \dot{v}^{inc} \cdot n ,  \dot{\mathbf{w}}|_\Gamma \rangle_{\Gamma\times \mathbb{R}^{+}\!\!,\sigma} \!+\! \langle \partial_{n}^{+} v^{inc}, \dot{w} \rangle_{\Gamma\times \mathbb{R}^{+}\!\!,\sigma} \ .\nonumber
\end{flalign}

Let $Z_{h,(\!\triangle t)}=V_{h,\Delta t}^{1,2}(\Omega)^3\!\!\subset\! H_{\sigma}^{1}(\!\mathbb{R}_{+}, H^{1}(\!\Omega\!)\!)^{3}$, $Y_{h,(\!\triangle t)}= V_{h,\Delta t}^{1,2}\!\!\subset\!H_{\sigma}^{1}(\!\mathbb{R}_{+}, H^{1/2}(\!\Gamma)\!)$, $X_{h,(\triangle t)}= V_{h,\Delta t}^{0,1} \textcolor{black}{\subset} H_{\sigma}^{1}(\!\mathbb{R}_{+}, H^{-1/2}(\!\Gamma)\!)$ be the conforming discretization spaces from \ref{SECT:retintop} in $\Omega$, resp.~$\Gamma$, based on tensor products of piecewise polynomial functions on a quasi-uniform
mesh in space and a uniform mesh in time.  Let $\widetilde{X}_{h,(\Delta t)}:=Z_{h,(\!\triangle t)} \times Y_{h,(\!\triangle t)} \times X_{h,(\triangle t)}$. \textcolor{black}{Note that the discretization order is higher in time than in space, in order to be  conforming. This corresponds to the loss of one time derivative for the boundary integral operators in Theorem 6, a well-known sub-optimal aspect of the standard functional analytic framework of space-time anisotropic Sobolev spaces. }\\

Then the discrete formulation reads:\\

\noindent Find $(\mathbf{u},\phi,\lambda) \in \widetilde{X}_{h,(\Delta t)}$ such that for all $({\mathbf{w}},w,m) \in \widetilde{X}_{h,(\Delta t)}$
\begin{flalign}
 &a(\mathbf{u},\dot{\mathbf{w}}) + \langle \dot{\phi}, \dot{\mathbf{w}}|_\Gamma \cdot n \rangle_{\Gamma\times \mathbb{R}^{+}\!\!,\sigma} \!-\! \langle \dot{{\mathbf{u}}}|_\Gamma \cdot n , \dot{w} \rangle_{\Gamma\times \mathbb{R}^{+}\!\!,\sigma} \!-\! \langle W \phi , \dot{w} \rangle_{\Gamma\times \mathbb{R}^{+}\!\!,\sigma} \nonumber
 \\ &\qquad - \textstyle \langle (\frac{1}{2} I - K^{T}) \lambda , \dot{w} \rangle_{\Gamma\times \mathbb{R}^{+}\!\!,\sigma} + \!\textstyle \langle (\frac{1}{2} I - K) \phi , \dot{m} \rangle_{\Gamma\times \mathbb{R}^{+}\!\!,\sigma} \!+\! \langle V \lambda , \dot{m} \rangle_{\Gamma\times \mathbb{R}^{+}\!\!,\sigma} \label{eq:FSIh}\\&\qquad \qquad =\! - \langle \dot{v}^{inc} \cdot n ,  \dot{\mathbf{w}}|_\Gamma \rangle_{\Gamma\times \mathbb{R}^{+}\!\!,\sigma} \!+\! \langle \partial_{n}^{+} v^{inc}, \dot{w} \rangle_{\Gamma\times \mathbb{R}^{+}\!\!,\sigma} \ .\nonumber
\end{flalign}
 
\textcolor{black}{Practical computations use $\sigma=0$.  See \cite{jr} for a detailed analysis of the role of the weight $\sigma$.\\}

In order to prove the well-posedness of the discrete formulation, we show the equivalence to a coercive formulation.
\begin{proposition}\label{prop:FSISayas41}
Let $({\mathbf{u}},\phi,\lambda) \in \widetilde{X}_{h,(\Delta t)}$ be a solution to the weak formulation \eqref{eq:FSIh}. 
Then with 
\begin{equation}\label{prop1eq4singledouble}
 v= D \phi - S \lambda\ , 
\end{equation}
 $({\mathbf{u}},v) \in Z_{h,(\triangle t)} \times H_{\sigma}^{1}(\mathbb{R}_{+}, H^{1}(\mathbb{R}^{3}\backslash \Gamma))$ satisfies for all $(\mathbf{w},w,m) \in \widetilde{X}_{h,(\Delta t)}$:
\begin{subequations}\label{eq:prop12equivtrans}
\begin{equation}\label{prop1eq5}
\textcolor{black}{\textstyle a({\mathbf{u}},{\dot{\mathbf{w}}}) + \langle  \lsem \dot{v}|_\Gamma \rsem  + \dot{v}^{inc} , \dot{\mathbf{w}}|_\Gamma \cdot n  \rangle_{\Gamma\times \mathbb{R}^{+}\!,\sigma} = 0 }
\end{equation}
\begin{equation}\label{prop1eq6}
\textstyle - \Delta v +\ddot v = 0 \quad \text{in} \ \mathbb{R}^{3} \backslash \Gamma
\end{equation}
\begin{equation}\label{prop1eq7}
 \lsem  v|_\Gamma \rsem \in Y_{h,(\triangle t)}, \ \lsem \partial_n v \rsem \in X_{h,(\triangle t)}
\end{equation}
\begin{equation}\label{prop1eq9}
\textcolor{black}{
\textstyle
-\langle  \dot{{\mathbf{u}}}|_\Gamma \cdot n , \dot{w} \rangle_{\Gamma \times \mathbb{R}^{+}\!,\sigma} - \langle \partial_n^{+} v , \dot{w} \rangle_{\Gamma \times \mathbb{R}^{+}\!,\sigma} = \langle {\partial_n^{+}}v^{inc} , \dot{w} \rangle_{\Gamma \times \mathbb{R}^{+}\!,\sigma}
} 
\end{equation}
\begin{equation}\label{prop1eq10}
\langle v^{\textcolor{black}{-}}|_\Gamma , \dot{m} \rangle_{\Gamma \times \mathbb{R}^{+}\!,\sigma} = 0 \quad \forall m \in X_{h,(\triangle t)} 
\end{equation}
\end{subequations}
where $ \lsem v|_\Gamma \rsem$, $ \lsem \partial_n v \rsem$ denote the jump of $v$, resp.~$\partial_n v$ across $\Gamma$.\\
Conversely, if $({\mathbf{u}},\phi,\lambda) = ({\mathbf{u}}, \lsem v|_\Gamma \rsem ,  \lsem \partial_n v \rsem)\in  \widetilde{X}_{h,(\triangle t)}$ satisfies \eqref{eq:prop12equivtrans} then the weak formulation \eqref{eq:FSIh} and \eqref{prop1eq4singledouble} hold.
\end{proposition}
\begin{proof}
Let $({\mathbf{u}},\phi,\lambda) \in \widetilde{X}_{h,(\triangle t)}$  fulfill the weak formulation \eqref{eq:FSIh}. Setting $v=D \phi - S \lambda$, the wave equation \eqref{prop1eq6} holds outside $\Gamma$.
Going onto the boundary with the jump relations \textcolor{black}{(see \ref{SECT:retintop})} 
\begin{align*}
v^+|_\Gamma &= (D \phi)^+|_\Gamma - (S \lambda)^+|_\Gamma = \textstyle (\frac{1}{2} I + K) \phi - V \lambda \ , \\
v^-|_\Gamma &= (D \phi)^-|_\Gamma -(S \lambda)^-|_\Gamma = \textstyle (-\frac{1}{2} I + K) \phi - V \lambda \ ,
\end{align*}
we obtain
\begin{equation}\label{eq:proofofprop1eq1}
  \lsem  v|_\Gamma \rsem =v^+|_\Gamma - v^-|_\Gamma =\phi \in Y_{h,(\triangle t)}\ ,
\end{equation}
and therefore the first assertion in \eqref{prop1eq7}.
\textcolor{black}{From \eqref{eq:FSIh} we see that 
\begin{equation}\label{erklarung8e}
\textstyle a({\mathbf{u}},{\dot{\mathbf{w}}}) + \langle \dot{\phi},\dot{\mathbf{w}}|_\Gamma \cdot n \rangle = - \langle \dot{v}^{inc}\cdot n  , \dot{\mathbf{w}}|_\Gamma  \rangle_{\Gamma\times \mathbb{R}^{+}\!,\sigma}
\end{equation}
and also \begin{equation*}
\langle v^{-}|_{\Gamma}, \dot{m} \rangle_{\Gamma\times\mathbb{R}^{+}} = \langle (\tfrac{1}{2} I -K)\phi , \dot{m} \rangle_{\Gamma\times\mathbb{R}^{+}} + \langle V \lambda , \dot{m} \rangle_{\Gamma\times\mathbb{R}^{+}} = 0 \ .
\end{equation*}
Equation \eqref{prop1eq10} follows.}
\textcolor{black}{Using \eqref{eq:proofofprop1eq1},  
we also obtain} \eqref{prop1eq5}. 

The jump relations give for $v$ in \eqref{prop1eq4singledouble} :
\begin{align*}
 \partial_n^{+} v &= \partial_n^{+}(D \phi) - \partial_n^{+}(S \lambda) = W \phi - K^{T} \lambda +{\textstyle \frac{1}{2}} I \lambda \ , \\
 \partial_n^{-} v &= \partial_n^{-}(D \phi) - \partial_n^{-}(S \lambda) = W \phi - K^{T} \lambda -{\textstyle \frac{1}{2}} I \lambda \ .
\end{align*}
Hence
\begin{equation*}
  \lsem \partial_n v \rsem = \partial_n^{+} v - \partial_n^{-} v = 
  \lambda\ , 
\end{equation*}
and thus $ \lsem \partial_n v \rsem \in X_{h,(\triangle t)} $, which establishes \eqref{prop1eq7}.

Finally, from  \eqref{eq:FSIh} with
\begin{equation*}
 - \partial_n^{+} v = - W \phi + K^{T} \lambda -{\textstyle \frac{1}{2}} I \lambda \ ,
\end{equation*}
\eqref{prop1eq9} holds.
Altogether \eqref{eq:prop12equivtrans} holds.\\ \ \\
To show the converse direction, we define $({\mathbf{u}},\phi,\lambda) := ({\mathbf{u}}, \lsem v|_\Gamma \rsem , \lsem \partial_n v \rsem) \in Z_{h,(\triangle t)} \times Y_{h,(\triangle t)} \times X_{h,(\triangle t)}$,
where ${\mathbf{u}}$ and $v$ fulfill \eqref{eq:prop12equivtrans}.
Since $v$ satisfies the wave equation \eqref{prop1eq6}, we get \eqref{prop1eq4singledouble} from the representation formula:
\begin{equation*}
 v= D \lsem v|_\Gamma \rsem - S \lsem \partial_n v \rsem   = D \phi - S \lambda \ .
\end{equation*}
Combining \eqref{prop1eq9} with the jump relations for $\partial_n^{+} v$ and $v|_\Gamma$, as well as setting $\phi$ in the equation \eqref{prop1eq5}, we obtain \eqref{eq:FSIh}.
\end{proof}
\begin{proposition}\label{prop:equivformfsips}
Let $$\widetilde{Z}_{h,(\triangle t)} \!\! = \!\! \lbrace \! v \! \in \! H_{\sigma}^{1}\!(\mathbb{R}_{+},\!H^{1} \!(\mathbb{R}^{3} \backslash  \Gamma)) \! : \! \lsem v|_\Gamma \rsem \! \in \! Y_{h,(\triangle t)}, \! \langle \!v^{-}|_\Gamma \! , \! \dot{m} \! \rangle_{\Gamma \times \mathbb{R}^{+}} \! =  \! 0 \ \forall m \in X_{h,(\triangle t)} \rbrace\ .$$ 
Then Problem \eqref{eq:prop12equivtrans} is equivalent to: \\
Find $ ({\mathbf{u}},v) \in Z_{h,(\triangle t)} \times \widetilde{Z}_{h,(\triangle t)} $ such that 
\begin{equation}\label{bilinear formfsi}
  \mathcal{A}(({\mathbf{u}},v),(\dot{\mathbf{w}},\dot{w})) = f((\dot{\mathbf{w}},\dot{w})) \quad \forall ({\mathbf{w}},w) \in Z_{h,(\triangle t)} \times \widetilde{Z}_{h,(\triangle t)} \ ,
\end{equation}
where
\begin{flalign*}
 &\mathcal{A}(({\mathbf{u}},v),({\mathbf{w}},w)) := (\tilde{\sigma}({\mathbf{u}}), \varepsilon({\mathbf{w}}))_{\Omega\times \mathbb{R}^{+}\!,\sigma} + (\ddot{{\mathbf{u}}}, {\mathbf{w}})_{\Omega\times \mathbb{R}^{+}\!,\sigma} +
(\nabla v , \nabla w)_{\mathbb{R}^{3} \backslash \Gamma \times \mathbb{R}^{+}\!,\sigma} \\ &+ 
(\ddot{v}, w)_{\mathbb{R}^{3} \backslash \Gamma \times \mathbb{R}^{+}\!,\sigma} 
- \langle \dot{{\mathbf{u}}}|_\Gamma \cdot n , \lsem w|_\Gamma \rsem \rangle_{\Gamma\times \mathbb{R}^{+}\!,\sigma} + \langle { \lsem  \dot{v}|_\Gamma \rsem} ,{\mathbf{w}|_\Gamma} \cdot n\rangle_{\Gamma\times \mathbb{R}^{+}\!,\sigma} 
\end{flalign*}
and
\begin{equation*}
f(({\mathbf{w}},w)) := - \langle \dot{v}^{inc}, {\mathbf{w}|_\Gamma} \cdot n\rangle_{\Gamma\times \mathbb{R}^{+}\!,\sigma} + \langle{
\textstyle \textcolor{black}{\partial_n^{+}v^{inc}} }, \lsem  w|_\Gamma \rsem \rangle_{\Gamma\times \mathbb{R}^{+}\!,\sigma} \ .
\end{equation*}
\end{proposition}
\begin{proof}
First, we assume that \textcolor{black}{\eqref{eq:FSIh}} 
holds with $({\mathbf{u}},\phi,\lambda) \in Z_{h,(\triangle t)} \times Y_{h,(\triangle t)} \times X_{h,(\triangle t)} $.
Since \eqref{prop1eq7} and \eqref{prop1eq10} hold, we know that $({\mathbf{u}},v) \in Z_{h,(\triangle t)} \times \widetilde{Z}_{h,(\triangle t)}$.
Now for all $w \in \widetilde{Z}_{h,(\triangle t)} $ using  the second relation in \eqref{prop1eq7} and $\langle \lsem \partial_n v \rsem , w^{-}|_\Gamma \rangle_{\Gamma \times \mathbb{R}^{+}} = 0 $
, Green's formula and \eqref{prop1eq6} lead to
\begin{align*}
 -\langle \partial_n^{+} v & , \lsem \dot{w}|_\Gamma \rsem \rangle_{\Gamma\times \mathbb{R}^{+}\!,\sigma}
 = \langle \partial_{n}^{+} v , \dot{w}^{-}|_\Gamma \rangle_{\Gamma \times \mathbb{R}^{+}\!,\sigma} - \langle  \partial_{n}^{+} v , \dot{w}^{+}|_\Gamma \rangle_{\Gamma \times \mathbb{R}^{+}\!,\sigma} \\
 & \textcolor{black}{= \langle \partial_n^{-} v , \dot{w}^{-}|_\Gamma \rangle_{\Gamma\times \mathbb{R}^{+}} - \langle \partial_n^{-} v , \dot{w}^{-}|_\Gamma \rangle_{\Gamma\times \mathbb{R}^{+}} + \langle \partial_n^{+} v , \dot{w}^{-}|_\Gamma \rangle_{\Gamma\times \mathbb{R}^{+}} - \langle \partial_n^{+} v , \dot{w}^{+}|_\Gamma \rangle_{\Gamma\times \mathbb{R}^{+}} } \\
 &= \langle \partial_n^{-} v , \dot{w}^{-}|_\Gamma \rangle_{\Gamma\times \mathbb{R}^{+}\!,\sigma} - \langle \partial_n^{+} v , \dot{w}^{+}|_\Gamma \rangle_{\Gamma\times \mathbb{R}^{+}\!,\sigma} + \langle \lsem \partial_n v \rsem , \dot{w}^{-}|_\Gamma \rangle_{\Gamma\times \mathbb{R}^{+}\!,\sigma} \\
 &= \langle \partial_n^{-} v , \dot{w}^{-}|_\Gamma \rangle_{\Gamma\times \mathbb{R}^{+}\!,\sigma} - \langle \partial_n^{+} v , \dot{w}^{+}|_\Gamma \rangle_{\Gamma\times \mathbb{R}^{+}\!,\sigma} \\
 &= (\nabla v , \nabla \dot{w})_{\Omega\times \mathbb{R}^{+},\sigma} + (\Delta v , \dot{w})_{\Omega\times \mathbb{R}^{+},\sigma} + (\nabla v , \nabla \dot{w})_{\Omega^{c}\times \mathbb{R}^{+}\!,\sigma} + (\Delta v , \dot{w})_{\Omega^{c}\times \mathbb{R}^{+}\!,\sigma} \\
 &= (\nabla v , \nabla \dot{w})_{\mathbb{R}^{3} \backslash \Gamma \times \mathbb{R}^{+}\!,\sigma} + ( \ddot v , \dot{w})_{(\mathbb{R}^{3}\backslash \Gamma) \times \mathbb{R}^{+}\!,\sigma} \ .
\end{align*}
Therefore testing \eqref{prop1eq9} with $\lsem \dot{w}|_\Gamma \rsem $ for $ w \in \widetilde{Z}_{h,(\triangle t)} $, 
\begin{equation*}
 - \langle\dot{\mathbf{u}}|_\Gamma \cdot n + \partial_n^{+} v + \partial_n^{\textcolor{black}{+}} v^{inc}, \lsem \dot{w}|_\Gamma \rsem \rangle_{\Gamma\times \mathbb{R}^{+}\!,\sigma} = 0 \ ,
\end{equation*}
we get for all $ \ w \in \widetilde{Z}_{h,(\triangle t)}$
\begin{equation}\label{eq:wewillneeditlaterprop2proof1}
 - \langle  {\dot {\mathbf{u}}}|_\Gamma   \cdot  n , \lsem \dot{w}|_\Gamma \rsem \rangle_{\Gamma\times \mathbb{R}^{+}\!,\sigma} \!\!+\! (\! \nabla v ,\! \nabla \dot{w})_{\mathbb{R}^{3} \backslash \Gamma \times \mathbb{R}^{+}\!,\sigma}
 \!\!+\! (\ddot v, \!\dot{w})_{(\mathbb{R}^{3}\backslash \Gamma)  \times \mathbb{R}^{+}\!,\sigma} \!\!= \! \langle \partial_{n}^{\textcolor{black}{+}} v^{inc} , \lsem \dot{w}|_\Gamma \rsem \rangle_{\Gamma \times \mathbb{R}^{+}\!,\sigma} .
 \end{equation}
Adding up \eqref{eq:wewillneeditlaterprop2proof1} and \eqref{prop1eq5} yields
$\mathcal{A}(({\mathbf{u}},v),(\dot{\mathbf{w}}, \dot{w})) = f((\dot{\mathbf{w}},\dot{w}))$. \\ \ \\
Conversely, assume \eqref{bilinear formfsi} holds. 
\textcolor{black}{Using \eqref{bilinear formfsi} for a test function $w \in \widetilde{Z}_{h,(\triangle t)} $, with compact support in $\mathbb{R}^{3} \backslash \Gamma$ we obtain the equation \eqref{eq:wewillneeditlaterprop2proof1}:}
\begin{align*}
 (\ddot v, \dot{w})_{\mathbb{R}^{3} \times \mathbb{R}^{+}\!,\sigma} \!+\! (\nabla v , \nabla \dot{w})_{\mathbb{R}^{3} \backslash \Gamma \times \mathbb{R}^{+}\!,\sigma} \!-\! \langle  \dot{{\mathbf{u}}}|_\Gamma \cdot n , \lsem \dot{w}|_\Gamma \rsem \rangle_{\Gamma \times \mathbb{R}^{+}\!,\sigma} \!=\! \langle \partial_{n}^{\textcolor{black}{+}} v^{inc} , \lsem \dot{w}|_\Gamma \rsem \rangle_{\Gamma \times \mathbb{R}^{+}\!,\sigma} .
\end{align*}
Using integration by parts:
\begin{flalign*}
 (\ddot v, \dot{w})_{\mathbb{R}^{3} \times \mathbb{R}^{+}\!,\sigma} &+ (\partial_n^{-} v , \dot{w}^{-}|_\Gamma)_{\Gamma \times \mathbb{R}^{+}\!,\sigma} - (\Delta v , \dot{w})_{\mathbb{R}^{3} \backslash \Gamma \times \mathbb{R}^{+}\!,\sigma} \\
  &- \langle \partial_n^{+} v, \dot{w}^{+}|_\Gamma \rangle_{\Gamma \times \mathbb{R}^{+}\!,\sigma} - \langle {\dot{\mathbf{u}}}|_\Gamma \cdot n , \lsem \dot{w}|_\Gamma \rsem \rangle_{\Gamma \times \mathbb{R}^{+}\!,\sigma} = \langle \partial_{n}^{+} v^{inc} , \lsem \dot{w}|_\Gamma \rsem \rangle_{\Gamma \times \mathbb{R}^{+}\!,\sigma}  . &
\end{flalign*}
Since $v$ satisfies the wave equation on the support of $w$ in $ \mathbb{R}^{3} \backslash \Gamma$.
\begin{equation*}
(\ddot v- \Delta v , \dot{w})_{\mathbb{R}^{3} \backslash \Gamma \times \mathbb{R}^{+}\!,\sigma}  = 0 \ .
\end{equation*}
Equation \eqref{prop1eq6} follows. Next
\begin{align*}
 \langle \partial_n^{-} v , \dot{w}^{-}|_\Gamma \rangle_{\Gamma \times \mathbb{R}^{+}\!,\sigma} - \langle \partial_n^{+} v , \dot{w}^{+}|_\Gamma \rangle_{\Gamma \times \mathbb{R}^{+}\!,\sigma} - \langle  \dot{\mathbf{u}}|_\Gamma \cdot n , \lsem \dot{w}|_\Gamma \rsem \rangle_{\Gamma \times \mathbb{R}^{+}\!,\sigma} =  \langle \partial_{n}^{+} v^{inc} , \lsem \dot{w}|_\Gamma \rsem \rangle_{\Gamma \times \mathbb{R}^{+}\!,\sigma} \ .
\end{align*}
Hence, for all  $w \in \tilde{Z}_{h,\Delta}$
\begin{align*}
 - \langle \partial_n^{+} v + {\dot{\mathbf{u}}}|_\Gamma \cdot n +\partial_n^{\textcolor{black}{+}} v^{inc}, \lsem \dot{w}|_\Gamma \rsem \rangle_{\Gamma \times \mathbb{R}^{+}\!,\sigma} - \langle \lsem \partial_n v \rsem , \dot{w}|_\Gamma^- \rangle_{\Gamma \times \mathbb{R}^{+}\!,\sigma} =  0 \ .
\end{align*}
Choosing $\lsem w|_\Gamma \rsem = 0$, we get the second relation in \eqref{prop1eq7} because
\begin{equation*}
 \langle \lsem \partial_n v \rsem , \dot{w}|_\Gamma^- \rangle_{\Gamma \times \mathbb{R}^{+}\!,\sigma} = 0 .
\end{equation*}
Second, choose $w \in \widetilde{Z}_{h,(\triangle t)}$ such that $w|_\Gamma^- = 0 $ yields 
\begin{equation*}
 - \langle \partial_n^{+} v + {\dot{\mathbf{u}}}|_\Gamma \cdot n + \partial_n^{\textcolor{black}{+}} v^{inc}, \lsem \dot{w}|_\Gamma \rsem \rangle_{\Gamma \times \mathbb{R}^{+}\!,\sigma} = 0 \ .
\end{equation*}
and hence 
\eqref{prop1eq9}, since $\lsem w|_\Gamma \rsem \in Y_{h,(\triangle t)}$.
From the definition of $ \widetilde{Z}_{h,(\triangle t)} $ we already get \eqref{prop1eq10} and \eqref{prop1eq7}.

Finally, we obtain the equation \eqref{prop1eq5} from the remaining terms in \eqref{bilinear formfsi}. 
\end{proof}

An analogous assertion to Propositions \ref{prop:FSISayas41} and \ref{prop:equivformfsips} holds for the continuous problem, instead of the finite element discretization.\\

We now aim to prove  coercivity of the problem \textcolor{black}{\eqref{bilinear formfsi}} for $({\mathbf{u}},v)$ in a suitable norm, defined as:
\begin{align*}
 \vertiii{({\mathbf{u}},v)}^{2} = (\tilde{\sigma}({\mathbf{u}}), \varepsilon({\mathbf{u}}))_{\Omega \times \mathbb{R}^{+}\!,\sigma} + (\dot{{\mathbf{u}}},\dot{{\mathbf{u}}})_{\Omega \times \mathbb{R}^{+}\!,\sigma} +(\nabla v,\nabla v)_{\mathbb{R}^{3}\backslash \Gamma \times \mathbb{R}^{+}\!,\sigma}
 +(\dot{v},\dot{v})_{\mathbb{R}^{3}\backslash \Gamma \times \mathbb{R}^{+}\!,\sigma} \ .
\end{align*}
Note that
\begin{flalign*}
 &\mathcal{A}(({\mathbf{u}},v),(\dot{{\mathbf{u}}}, \dot{v})) = (\tilde{\sigma}({\mathbf{u}}), \varepsilon(\dot{{\mathbf{u}}}))_{\Omega \times \mathbb{R}^{+}\!,\sigma} + (\ddot{{\mathbf{u}}}, \dot{{\mathbf{u}}})_{\Omega\times \mathbb{R}^{+}\!,\sigma} +
(\nabla v , \nabla \dot{v})_{\mathbb{R}^{3} \backslash \Gamma \times \mathbb{R}^{+}\!,\sigma} 
\\ &+ (\ddot{v}, \dot{v})_{\mathbb{R}^{3} \backslash \Gamma \times \mathbb{R}^{+}\!,\sigma} 
+ \langle {\dot{\mathbf{u}}}|_\Gamma \cdot n , \lsem \gamma \dot{v} \rsem \rangle_{\Gamma \times \mathbb{R}^{+}\!,\sigma} - \langle \lsem \textcolor{black}{\dot{v}|_{\Gamma}} \rsem , {\dot{\mathbf{u}}}|_\Gamma \cdot n\rangle_{\Gamma \times \mathbb{R}^{+}\!,\sigma} \\
&= (\tilde{\sigma}({\mathbf{u}}), \varepsilon(\dot{{\mathbf{u}}}))_{\Omega \times \mathbb{R}^{+}\!,\sigma} + (\ddot{{\mathbf{u}}}, \dot{{\mathbf{u}}})_{\Omega \times \mathbb{R}^{+}\!,\sigma} +
(\nabla v , \nabla \dot{v})_{\mathbb{R}^{3} \backslash \Gamma \times \mathbb{R}^{+}\!,\sigma} + 
(\ddot{v}, \dot{v})_{\mathbb{R}^{3} \backslash \Gamma \times \mathbb{R}^{+}\!,\sigma} \\
&= \int_{0}^{\infty} \int_{\Omega} \tilde{\sigma}({\mathbf{u}}) : \varepsilon(\dot{{\mathbf{u}}}) dx e^{-2 \sigma t} dt 
+ \int_{0}^{\infty} \int_{\Omega} \ddot{{\mathbf{u}}} \dot{{\mathbf{u}}} dx e^{-2 \sigma t} dt \\
&+ \int_{0}^{\infty} \int_{\mathbb{R}^{3}\backslash\Gamma} \nabla v  \nabla \dot{v} dx e^{-2 \sigma t} dt
+ \int_{0}^{\infty} \int_{\mathbb{R}^{3}\backslash\Gamma} \ddot{v}  \dot{v} dx e^{-2 \sigma t} dt \\ 
&= \int_{0}^{\infty} \tfrac{1}{2} \partial_t (\int_{\Omega} \tilde{\sigma}({\mathbf{u}}) : \varepsilon({\mathbf{u}}) dx) e^{-2 \sigma t} dt + \int_{0}^{\infty} \tfrac{1}{2} \partial_t (\int_{\Omega}  \dot{{\mathbf{u}}}^2 dx ) e^{-2 \sigma t} dt \\
&+ \int_{0}^{\infty} \tfrac{1}{2} \partial_t (\int_{\mathbb{R}^{3}\backslash\Gamma}  (\nabla v)^2  dx ) e^{-2 \sigma t} dt
+ \int_{0}^{\infty} \tfrac{1}{2} \partial_t (\int_{\mathbb{R}^{3}\backslash\Gamma} \dot{v}^2 dx ) e^{-2 \sigma t} dt \ .
\end{flalign*}Using integration by parts in time,  the zero initial condition and $\sigma>0$, 
we obtain the  coercivity estimate:
\begin{flalign*}
&\mathcal{A}(\!({\mathbf{u}},v),(\dot{{\mathbf{u}}}, \dot{v})\!) \!= \textcolor{black}{- \tfrac{1}{2} \!\! \int_{0}^{\infty} \!\!\Big( \!\int_{\Omega} \!\! \tilde{\sigma}({\mathbf{u}}) \! : \! \varepsilon({\mathbf{u}}) dx \Big) \partial_t (e^{-2 \sigma t}) dt
 -\!\tfrac{1}{2} \!\! \int_{0}^{\infty} \!\!\Big( \!\int_{\Omega}  \dot{{\mathbf{u}}}^2 dx \Big) \partial_t(e^{-2 \sigma t}) dt} \\
 &-\tfrac{1}{2} \int_{0}^{\infty} ( \int_{\mathbb{R}^{3}\backslash\Gamma}   (\nabla v)^2 dx ) \partial_t(e^{-2 \sigma t}) dt
 -\tfrac{1}{2} \int_{0}^{\infty} ( \int_{\mathbb{R}^{3}\backslash\Gamma}   \dot{v}^2 dx  )\partial_t(e^{-2 \sigma t}) dt \\ 
&= \sigma (\int_{0}^{\infty} ( \int_{\Omega} \tilde{\sigma}({\mathbf{u}}) : \varepsilon({\mathbf{u}}) dx) e^{-2 \sigma t} dt
 + \int_{0}^{\infty} ( \int_{\Omega}  \dot{{\mathbf{u}}}^2 dx ) e^{-2 \sigma t} dt \\ 
 &+ \int_{0}^{\infty} ( \int_{\mathbb{R}^{3}\backslash\Gamma}  (\nabla v)^2 dx ) e^{-2 \sigma t} dt
 + \int_{0}^{\infty} ( \int_{\mathbb{R}^{3}\backslash\Gamma} \dot{v}^2 dx  ) e^{-2 \sigma t} dt ) \\
&= \sigma \vertiii{({\mathbf{u}},v)}^{2} \textcolor{black}{\gtrsim_{\sigma}} \lVert {\mathbf{u}} \rVert_{0,1,\Omega}^{2} + \lVert v \rVert_{0,1,\Omega^{c}}^{2} \ .
\end{flalign*}
This implies, in particular, uniqueness of \textcolor{black}{the solution \eqref{bilinear formfsi} and therefore also the} solutions to \eqref{eq:FSIsolkap4} and  \eqref{eq:FSIh}.

\section{A priori error estimate}\label{sec3}
We state an a priori error estimate:
\begin{theorem}\label{mainthm}
 Let $(\mathbf{u},\phi,\lambda) \in \widetilde{X}$ satisfy \eqref{eq:FSIsolkap4} and $(\mathbf{u}_{h},\phi_{h},\lambda_{h}) \in \widetilde{X}_{h,\Delta t} $ satisfy \eqref{eq:FSIh}. Then 
 \begin{align*}
  &\lVert \mathbf{u} - \mathbf{u}_{h} \rVert_{0,1,\Omega}^{2} + \lVert \phi - \phi_{h} \rVert_{0,1/2,\Gamma}^{2} + \lVert \lambda-\lambda_{h} \rVert_{0,-1/2,\Gamma}^{2}\lesssim_{\sigma} \\ &\inf_{(\mathbf{w}_{h},\psi_{h},\mu_{h}) \in\widetilde{X}_{h,\Delta t}} (1+\!\!\frac{1}{(\Delta t)^{2}}) \lVert \mathbf{u} \!-\! \mathbf{w}_{h} \rVert_{1,1,\Omega}^{2} \!+\!(\!1\!+\!\!\frac{1}{(\Delta t)^{2}}) \lVert \phi \!-\! \psi_{h} \rVert_{1,1/2,\Gamma}^{2} \!\!\! +\!(\!1\!+\!\!\frac{1}{(\Delta t)^{2}}) \lVert \lambda\!-\!\mu_{h} \rVert_{1,-1/2,\Gamma}^{2}.
 \end{align*}
\end{theorem}
\begin{proof}
 Let $(\mathbf{u},\phi,\lambda) \in \widetilde{X}$ satisfy \eqref{eq:FSIsolkap4} and $(\mathbf{u}_{h},\phi_{h},\lambda_{h}) \in \widetilde{X}_{h,\Delta t} $ satisfy \eqref{eq:FSIh}. Then for all $(\tilde{\mathbf{w}},\tilde{\phi},\tilde{\lambda}) \in \widetilde{X}_{h,\Delta t}$
 \begin{flalign*}
  &\lVert \mathbf{u} - \mathbf{u}_{h} \rVert_{0,1,\Omega}^{2} + \lVert \phi - \phi_{h} \rVert_{0,1/2,\Gamma}^{2} + \lVert \lambda-\lambda_{h} \rVert_{0,-1/2,\Gamma}^{2} &\\
  &\lesssim \lVert \mathbf{u} - \tilde{\mathbf{w}} \rVert_{0,1,\Omega}^{2} + \lVert \tilde{\mathbf{w}} - \mathbf{u}_{h} \rVert_{0,1,\Omega}^{2}+ \lVert \phi - \tilde{\phi} \rVert_{0,1/2,\Gamma}^{2} +  \lVert \tilde{\phi} - \phi_{h} \rVert_{0,1/2,\Gamma}^{2} & \\ &+ \lVert \lambda-\tilde{\lambda} \rVert_{0,-1/2,\Gamma}^{2}+ \lVert \tilde{\lambda}-\lambda_{h} \rVert_{0,-1/2,\Gamma}^{2} \ . &
 \end{flalign*}
We therefore  focus on estimates for $ \lVert \tilde{\mathbf{w}} - \mathbf{u}_{h} \rVert_{0,1,\Omega}^{2} + \lVert \tilde{\phi} - \phi_{h} \rVert_{0,1/2,\Gamma}^{2} +  \lVert \tilde{\lambda}-\lambda_{h} \rVert_{0,-1/2,\Gamma}^{2} $. 
Now using coercivity and Galerkin orthogonality \textcolor{black}{with $v:=D\phi-S \lambda$, $v_{h} := D \phi_{h} - S \lambda_{h}$ and $\tilde{r} := D \tilde{\phi}-S \tilde{\lambda}$ }
\begin{flalign*}
 & \lVert \tilde{\mathbf{w}} - \mathbf{u}_{h} \rVert_{0,1,\Omega}^{2} + \lVert \tilde{\phi} - \phi_{h} \rVert_{0,1/2,\Gamma}^{2} +  \lVert \tilde{\lambda}-\lambda_{h} \rVert_{0,-1/2,\Gamma}^{2} \\ & 
  \lesssim \lVert \tilde{\mathbf{w}} - \mathbf{u}_{h} \rVert_{0,1,\Omega}^{2} + \lVert \tilde{r} - v_{h} \rVert_{0,1,\mathbb{R}^{3}\backslash\Gamma}^{2} & \\
 & \lesssim_\sigma \vertiii{(\tilde{\mathbf{w}}-\mathbf{u}_{h},\tilde{r}-v_{h})}^{2}
 = \mathcal{A}\Big(\begin{pmatrix} \tilde{\mathbf{w}}-\mathbf{u}_{h} \\ \tilde{r} - v_{h}\end{pmatrix}^{\mathrm{T}}, \begin{pmatrix} \dot{\tilde{\mathbf{w}}}-\dot{\mathbf{u}}_{h}) \\  \dot{\tilde{r}} - \dot{v}_{h}) \end{pmatrix}^{\mathrm{T}}\Big)\ . & \\& = \mathcal{A}\Big(\begin{pmatrix} \tilde{\mathbf{w}}-\mathbf{u} \\ \tilde{r} - v\end{pmatrix}^{\mathrm{T}}, \begin{pmatrix} \dot{\tilde{\mathbf{w}}}-\dot{\mathbf{u}}_{h}) \\ \dot{\tilde{r}} -\dot{v}_{h}) \end{pmatrix}^{\mathrm{T}}\Big) +
  \mathcal{A}\Big(\begin{pmatrix} \mathbf{u}-\mathbf{u}_{h} \\ v-v_{h}\end{pmatrix}^{\mathrm{T}}, \begin{pmatrix} \dot{\tilde{\mathbf{w}}}-\dot{\mathbf{u}}_{h}\\ \dot{\tilde{r}} - \dot{v}_{h} \end{pmatrix}^{\mathrm{T}}\Big) \\
  &= \mathcal{A}\Big(\begin{pmatrix} \tilde{\mathbf{w}}-\mathbf{u} \\ \tilde{r} - v\end{pmatrix}^{\mathrm{T}}, \begin{pmatrix} \dot{\tilde{\mathbf{w}}}-\dot{\mathbf{u}}_{h} \\ \dot{\tilde{r}} - \dot{v}_{h}) \end{pmatrix}^{\mathrm{T}}\Big) \ .& 
 \end{flalign*}
By definition of the single and double layer potentials $S$ and $D$, the functions \textcolor{black}{$v,v_{h},\tilde{r}$} 
all satisfy the wave equation in $\mathbb{R}^3 \backslash \Gamma$.  Further, from  Proposition \ref{prop:FSISayas41}, $\lsem  v_{h}|_\Gamma \rsem, \lsem \tilde{r}|_\Gamma \rsem \in Y_{h,\Delta t}^0$, $\lsem \partial_{n} v_{h} \rsem , \lsem \partial_{n} \tilde{r} \rsem \in X_{h,\Delta t}^0$ and $\langle \textcolor{black}{{v}_{h}^{-}|_{\Gamma}} , m \rangle_{\Gamma \times \mathbb{R}^{+}} = 0 $,  $\langle \textcolor{black}{\tilde{r}^{-}|_{\Gamma}} , m \rangle_{\Gamma \times \mathbb{R}^{+}} = 0$. Using the definition of $\mathcal{A}$ and Green's theorem, we find
\begin{flalign*}
 & \lVert \tilde{\mathbf{w}} - \mathbf{u}_{h} \rVert_{0,1,\Omega}^{2} + \lVert \tilde{\phi} - \phi_{h} \rVert_{0,1/2,\Gamma}^{2} +  \lVert \tilde{\lambda}-\lambda_{h} \rVert_{0,-1/2,\Gamma}^{2} 
& \\
  &\lesssim_\sigma \int\limits_{0}^{\infty} e^{-2 \sigma t} \Bigg\lbrace \int\limits_{\Omega} \tilde{\sigma}(\tilde{\mathbf{w}}-\mathbf{u}) : \varepsilon(\dot{\tilde{\mathbf{w}}}-\dot{\mathbf{u}}_{h}) dx + \int_{\Omega}  (\ddot{\tilde{\mathbf{w}}}-\ddot{\mathbf{u}}) (\dot{\tilde{\mathbf{w}}}-\dot{\mathbf{u}}_{h}) dx & \\
  &+ \int\limits_{\mathbb{R}^{3}\backslash\Gamma} \nabla(\tilde{r}-v) \nabla(\dot{\tilde{r}}-\dot{v}_{h}) dx + \int\limits_{\mathbb{R}^{3}\backslash\Gamma} (\ddot{\tilde{r}}-\ddot{v}) (\dot{\tilde{r}}-\dot{v}_{h}) dx & \\
  &- \int\limits_{\Gamma} (\dot{\tilde{\mathbf{w}}}-\dot{\mathbf{u}})|_\Gamma\cdot n \lsem (\dot{\tilde{r}}-\dot{v}_{h})|_\Gamma \rsem ds_x + \int\limits_{\Gamma} \lsem (\dot{\tilde{r}}-\dot{v})|_\Gamma \rsem (\dot{\tilde{\mathbf{w}}}-\dot{\mathbf{u}}_{h})|_\Gamma \cdot n ds_x \Bigg\rbrace dt & \\
  &= \int\limits_{0}^{\infty} e^{-2 \sigma t} \Bigg\lbrace \int\limits_{\Omega} \tilde{\sigma}(\tilde{\mathbf{w}}-\mathbf{u}) : \varepsilon(\dot{\tilde{\mathbf{w}}}-\dot{\mathbf{u}}_{h}) dx + \int_{\Omega} (\ddot{\tilde{\mathbf{w}}}-\ddot{\mathbf{u}}) (\dot{\tilde{\mathbf{w}}}-\dot{\mathbf{u}}_{h}) dx & \\
  &+ \int\limits_{\Gamma} \partial_{n}^{+}(\tilde{r}-v) \lsem (\dot{\tilde{r}}-\dot{v}_{h})|_\Gamma \rsem ds_x 
  - \int\limits_{\Gamma}(\dot{\tilde{\mathbf{w}}}- \dot{\mathbf{u}})|_\Gamma\cdot n \lsem (\dot{\tilde{r}}-\dot{v}_{h})|_\Gamma \rsem ds_x \\ &+ \int\limits_{\Gamma} \lsem (\dot{\tilde{r}}-\dot{v})|_\Gamma \rsem(\dot{\tilde{\mathbf{w}}}-\dot{\mathbf{u}}_{h})|_\Gamma \cdot n ds_x \Bigg\rbrace dt  &
    \end{flalign*}
 \begin{flalign*}
   &= \int\limits_{0}^{\infty} e^{-2 \sigma t} \Bigg\lbrace \int\limits_{\Omega} \tilde{\sigma}(\tilde{\mathbf{w}}-\mathbf{u}) : \varepsilon(\dot{\tilde{\mathbf{w}}}-\dot{\mathbf{u}}_{h})) dx + \int_{\Omega} (\ddot{\tilde{\mathbf{w}}}-\ddot{\mathbf{u}}) (\dot{\tilde{\mathbf{w}}}-\dot{\mathbf{u}}_{h}) dx & \\
  &+ \int\limits_{\Gamma} (W (\tilde{\phi}-\phi) - (K^{T}-\frac{1}{2}I) (\tilde{\lambda}-\lambda)) (\dot{\tilde{\phi}}-\dot{\phi}_{h}) ds_x & \\
  &- \int\limits_{\Gamma} (\dot{\tilde{\mathbf{w}}}-\dot{\mathbf{u}})|_\Gamma\cdot n (\dot{\tilde{\phi}}-\dot{\phi}_{h})) ds_x + \int\limits_{\Gamma} (\dot{\tilde{\phi}}-\dot{\phi}) (\dot{\tilde{\mathbf{w}}}-\dot{\mathbf{u}}_{h})|_\Gamma \cdot n ds_x \Bigg\rbrace dt  \ .& 
\end{flalign*}
We estimate the individual terms. 
Using Young's inequality we have for  $\epsilon>0$
\begin{flalign*}
  &\int\limits_{0}^{\infty} e^{-2 \sigma t} \Big\lbrace \int\limits_{\Omega} \tilde{\sigma}(\tilde{\mathbf{w}}-\mathbf{u}) : \varepsilon(\dot{\tilde{\mathbf{w}}}-\dot{\mathbf{u}}_{h})) dx + \int_{\Omega} (\ddot{\tilde{\mathbf{w}}}-\ddot{\mathbf{u}}) (\dot{\tilde{\mathbf{w}}}-\dot{\mathbf{u}}_{h}) dx \Big\rbrace dt & \\
  &\lesssim_{\sigma} \tfrac{1}{\epsilon (\Delta t)^{2}} \lVert \tilde{\mathbf{w}} - \mathbf{u} \rVert_{\textcolor{black}{1},1,\Omega}^{2} + \epsilon \lVert \tilde{\mathbf{w}}-\mathbf{u}_{\textcolor{black}{h}} \rVert_{0,1,\Omega}^{2} \ .
\end{flalign*}
Next we estimate
\begin{flalign*}
 &\int\limits_{0}^{\infty} e^{-2\sigma t} \int\limits_{\Gamma} (W (\tilde{\phi}-\phi) - (K^{T}-\frac{1}{2}I) (\tilde{\lambda}-\lambda)) (\dot{\tilde{\phi}}-\dot{\phi}_{h}) ds_x dt & \\
 &\lesssim ( \lVert W (\tilde{\phi}-\phi)  \rVert_{0,-1/2,\Gamma} + \lVert (K^{T}-\frac{1}{2}I) (\tilde{\lambda}-\lambda) \rVert_{0,-1/2,\Gamma}) \lVert \tilde{\phi}-\phi_{h} \rVert_{1,1/2,\Gamma} \ .
\end{flalign*}
Using the inverse estimate as in (3.182) in \cite{phdglafke} 
\begin{equation*}
 \lVert \tilde{\phi} \rVert_{1,1/2,\Gamma} \lesssim \frac{1}{\Delta t} \lVert \tilde{\phi} \rVert_{0,1/2,\Gamma} \ ,
\end{equation*}
we further estimate with the mapping properties of the integral operators 
\begin{flalign*}
 &\lesssim_{\sigma} ( \lVert W (\tilde{\phi}-\phi)  \rVert_{0,-1/2,\Gamma} + \lVert (K^{T}-\frac{1}{2}I) (\tilde{\lambda}-\lambda) \rVert_{0,-1/2,\Gamma}) \frac{1}{\Delta t}\lVert \tilde{\phi}-\phi_{h} \rVert_{0,1/2,\Gamma} \\
 &\lesssim \tfrac{1}{\epsilon (\Delta t)^{2}} \lVert \tilde{\phi}-\phi \rVert_{\textcolor{black}{1},1/2,\Gamma}^{2} + \tfrac{1}{\epsilon (\Delta t)^{2}} \lVert \tilde{\lambda} - \lambda \rVert_{1,-1/2,\Gamma}^{2} + \epsilon \lVert \tilde{\phi}-\phi_{h} \rVert_{0,1/2,\Gamma}^{2} \ .
\end{flalign*}
For the fourth term, we get:
\begin{flalign*}
 &\int_{0}^{\infty} e^{-2 \sigma t} \int_{\Gamma} (\dot{\tilde{\mathbf{w}}}-\dot{\mathbf{u}})|_\Gamma \cdot n (\dot{\tilde{\phi}}-\dot{\phi}_{h}) ds_x dt & \\
 &\lesssim_{\sigma} \lVert (\dot{\tilde{\mathbf{w}}}-\dot{\mathbf{u}})|_\Gamma \cdot n \rVert_{0,-1/2,\Gamma} \lVert  (\dot{\tilde{\phi}}-\dot{\phi}_{h}) \rVert_{0,1/2,\Gamma} \\
 &\lesssim \lVert \tilde{\mathbf{w}}-\mathbf{u} \rVert_{1,1/2,\Gamma} \lVert \tilde{\phi}-\phi_{h} \rVert_{1,1/2,\Gamma} \ 
 \lesssim \tfrac{1}{\epsilon (\Delta t)^{2}} \lVert \tilde{\mathbf{w}} - \mathbf{u} \rVert_{1,1,\Omega}^{2} + \epsilon \lVert \tilde{\phi} - \phi_{h} \rVert_{0,1/2,\Gamma}^{2} \ .
 \end{flalign*}
For the last term, analogously the trace theorem and the inverse estimate show:
\begin{flalign*}
 &\int_{0}^{\infty} \!\!\!\!e^{-2 \sigma t}\!\!\! \int_{\Gamma} (\dot{\tilde{\phi}}-\dot{\phi})  (\dot{\tilde{\mathbf{w}}}-\dot{\mathbf{u}}_{h})|_\Gamma \cdot n ds_x dt  
 \lesssim_{\sigma}\lVert \dot{\tilde{\phi}}-\dot{\phi}\rVert_{0,1/2,\Gamma} \lVert \dot{\tilde{\mathbf{w}}} - \dot{\mathbf{u}}_{h}|_\Gamma\!\cdot\! n \rVert_{0,-1/2,\Gamma} &\\
 & \lesssim \tfrac{1}{\Delta t} \lVert \tilde{\phi} - \phi \rVert_{1,1/2,\Gamma} \lVert \tilde{\mathbf{w}}-\mathbf{u}_{h} \rVert_{0,1,\Omega} 
 \lesssim \tfrac{1}{\epsilon(\Delta t)^{2}} \lVert \tilde{\phi} - \phi \rVert_{1,1/2,\Gamma}^{2} +  \epsilon \lVert \tilde{\mathbf{w}}-\mathbf{u}_{h} \rVert_{0,1,\Omega}^{2} \ .
\end{flalign*}
Moving the terms with positive powers of $\epsilon$ to the left hand side  \textcolor{black}{and choosing a fixed, sufficiently small $\epsilon>0$ depending on $\sigma$, we conclude:}
\begin{flalign*}
 &\lVert \tilde{\mathbf{w}} - \mathbf{u}_{h} \rVert_{0,1,\Omega}^{2} +\lVert \tilde{\phi} - \phi_{h}\rVert_{0,1/2,\Gamma}^{2} + \lVert \tilde{\lambda} - \lambda_{h} \rVert_{0,-1/2,\Gamma}^{2} & \\
 &\lesssim_{\sigma}  \textcolor{black}{\tfrac{1}{(\Delta t)^{2}} \lVert \tilde{\mathbf{w}} - \mathbf{u} \rVert_{1,1,\Omega}^{2}}  + \tfrac{1}{(\Delta t)^{2}} \lVert \tilde{\phi} - \phi \rVert_{1,1/2,\Gamma}^{2} + \tfrac{1}{(\Delta t)^{2}} \lVert \tilde{\lambda}-\lambda \rVert_{1,-1/2,\Gamma}^{2} \ ,
 \end{flalign*}
and therefore
\begin{flalign*}
 &\lVert \mathbf{u} - \mathbf{u}_{h} \rVert_{0,1,\Omega}^{2} +\lVert \phi - \phi_{h}\rVert_{0,1/2,\Gamma}^{2} + \lVert \lambda - \lambda_{h} \rVert_{0,-1/2,\Gamma}^{2} & \\
 &\lesssim_{\sigma} (1+\frac{1}{(\Delta t)^{2}}) \Big( \lVert \tilde{\mathbf{w}} - \mathbf{u} \rVert_{1,1,\Omega}^{2}+ \lVert \tilde{\phi} - \phi \rVert_{1,1/2,\Gamma}^{2} + \lVert \tilde{\lambda} - \lambda \rVert_{1,-1/2,\Gamma}^{2} \Big) \ .
\end{flalign*}
This proves the  assertion.
\end{proof}

\section{Numerical results}\label{sec4}

\textcolor{black}{This section presents numerical results for the fluid structure-interaction problem given by \eqref{eq:FSIsolkap4}, in 3d. While finite element discretizations of fluid-structure interaction have attracted significant recent interest,   coupled  finite and boundary element procedures in the time domain are only beginning to be explored. In 2d, numerical results have been presented in  \cite{MR3614885}, based on time discretization by convolution quadrature of the boundary integral operators. A similar appooach has been demonstrated for the interaction of waves with a thermoelastic solid in 2d \cite{thermo}. The authors are not aware of any related numerical results in the mathematical literature based on time-domain Galerkin boundary element methods, as presented in this work.  On the other hand, such methods are now actively being studied for wave-wave interaction in 2d and 3d, as in \cite{ajrt} and \cite{aimi1}.} \\

In the numerical experiments for Problem \eqref{eq:wholeproblematonce}, the variational formulation \eqref{eq:FSIsolkap4} is solved, \textcolor{black}{for $\sigma=0$,} by choosing as ansatz function in the interior domain $\Omega$
\begin{equation}\label{eq:discinnansatz}
{\mathbf{u}}_{h,\triangle t}(x,t) = \sum_{k=1}^{N_t} \sum_{\nu=1}^{3} \sum_{i=1}^{\textcolor{black}{N_o}} u_{\nu,i}^{k} \beta_{\Delta t}^{k}(t) {\bf e}_{\nu} \eta_{h}^{i}(x)\ ,
\end{equation}
where $\{\eta_{h}^{i}\}$ denotes the basis of piecewise linear hat functions for $V^{1}_{h}(\Omega)$ and 
$$
 \beta^{m}_{\Delta t}(t)=(\Delta t)^{-1}((t-t_{m-1}) \gamma^{m}_{\Delta t}(t)-(t-t_{m+1})\gamma^{m+1}_{\Delta t}(t)) \ .$$
Here $\gamma^{m}_{\Delta t}(t)=H(t-t_{m-1})-H(t-t_{m})$, with $H$ the Heaviside function. The test functions are given by 
\begin{equation}\label{eq:discinntest}
{\dot{\mathbf{w}}}_{h,\Delta t} = \eta_{h}^{l}(x) \gamma_{\Delta t}^{n}(t) {\bf e}_{\mu} , 
\end{equation}
for $l=1,\ldots,\textcolor{black}{N_o}$, $ n=1,\ldots,N_t$ and $\mu=1,2,3$.

The ansatz functions on $\Gamma$ are taken as
\begin{equation}\label{eq:discboundWansatz}
 \phi_{h,\triangle t}(x,t)= \sum_{m=1}^{N_t} \sum_{i=1}^{N_{{s'}}} \varphi^{m}_{i} \beta_{\Delta t}^{m}(t) \xi_{h}^{i}(x) 
\end{equation}
and
\begin{equation}\label{eq:discboundVansatz}
\lambda_{h,\triangle t}(x,t) = \sum\limits_{m=1}^{N_t} \sum\limits_{i=1}^{N_{s'}} \lambda_{i}^{m} \beta_{\Delta t}^{m}(t)  \xi_{h}^{i}(x), 
\end{equation}
where $\{\xi_{h}^{i}\}$ denotes the basis of piecewise linear hat functions for $V^{1}_{h}$.
The corresponding test functions are chosen as
\begin{equation}\label{eq:discboundWtest}
\dot{w}_{h,\Delta t}= \gamma_{\Delta t}^{n}(t) \xi_{h}^{j}(x) \ ,
\end{equation}
\begin{equation}\label{eq:discboundVtest}
m_{h,\Delta t}=\gamma_{\Delta t}^{n}(t)\xi_{h}^{j}(x) \ ,
\end{equation}
for $1\leq n \leq N_t$ and $1\leq j \leq N_{{s'}}$. \textcolor{black}{The resulting discretization of the Poincar\'{e}-Steklov operator has been tested in \cite{contact,graded}, and corresponding results are obtained for more natural discretizations with piecewise constant $\lambda_{h,\triangle t}$. Piecewise linear and higher order test functions are considered in \cite{gs}.} 

As shown in Appendix B, this discretization of  \eqref{eq:FSIsolkap4}  leads to a time-stepping scheme, which solves a  system of the following structure in each time step $\geq 3$:
\begin{align*}
 &\begin{pmatrix}
  \frac{(\triangle t)}{2} A + \frac{1}{(\triangle t)} M & [0,n_{x}\!RI]^{T} & 0 \\ [0,-RI\!n_{x}] & -W^{0} & {K^{T}}^{0}-\frac{1}{2} \frac{(\triangle t)}{2} I \\ 0 & -K^{0} - \frac{1}{2} I & V^{0} \end{pmatrix}
 \begin{pmatrix} u^{n} \\ \varphi^{n} \\ \lambda^{n} \end{pmatrix} \\
 &= \begin{pmatrix} H^{n} + H^{n-1} - A \frac{(\triangle t)}{2} u^{n-1} + M \frac{2}{(\triangle t)} u^{n-1} - M \frac{1}{(\triangle t)} u^{n-2} +n_{x}\!RI \varphi^{n-1}\\ G^{n} +G^{n-1} + RI\!n_{x} u^{n-1}_{\Gamma}   + \sum_{m=1}^{n-1} W^{n-m} \varphi^{m} - \sum_{m=1}^{n-1} {K^{T}}^{n-m} \lambda^{m} + \frac{1}{2} \frac{(\triangle t)}{2} I \lambda^{n-1} \\  \sum_{m=1}^{n-1} K^{n-m} \varphi^{m} - \frac{1}{2} I \varphi^{n-1}- \sum_{m=1}^{n-1} V^{n-m} \lambda^{m} \end{pmatrix}.
\end{align*}
The system in the first two time steps is similar, see Appendix B. We solve this system repeatedly until our desired time step $N_t$ is reached.\\

\noindent \textbf{Example.}
Let $\Omega = [-1,1]^3$. Using the discretization described above, we compute the solutions to the discrete system \eqref{eq:FSIsolkap4}  up to time $T=4$ for  data $v^{inc}$ corresponding to the exact solution 
\begin{equation}\label{exactu}
\mathbf{u}(x,t)=\begin{pmatrix} (\sin(\pi(t-\tfrac{x_{1}}{2})))^{5} (H(-1+t-\tfrac{x_{1}}{2})-H(-3+t-\tfrac{x_{1}}{2})) \\ 0 \\ 0 \end{pmatrix}\ ,
\end{equation}
\begin{equation}\label{exactv}
v(x,t)=\tfrac{|x|-t}{2|x|} (1+\cos(\tfrac{\pi(|x|-t)}{0.9})) H(0.9-||x|-t|)\ .
\end{equation}
 We use  uniform discretizations by tetrahedra as depicted in Figure \ref{fig:meshplot} and a time step $\Delta t$ such that $\frac{\Delta t}{h} \simeq 0.1414$. We denote the  number of grid points on an edge of the cube by $n+1$. The finest mesh is then given by $n=24$ and consists of $69120$ tetrahedra, corresponding to $\Delta t = 0.01667$. The convergence of the numerical solution to the exact solution is studied as the mesh is refined, and we measure the  error in terms of the $L^2$-norm in space, resp.~space-time. \\

Figure \ref{fig:FSIat-1-1-1} shows the first component ${\bf u}_1$ of the numerical and exact solutions at the corner point ${\bf x_0}=(-1,-1,-1)$ as a function of time for $n=2, 4,8$. The behaviour of the solution in $\Omega$, resp.~$\Omega^c$, is illustrated in Figure \ref{fig:FSIthirdplot}, which plots the $L^{2}$-norms of ${\bf u}$ \textcolor{black}{in $\Omega$}, resp.~of \textcolor{black}{$v|_\Gamma = \varphi$ on $\Gamma$}, as a function of time for $n=2, 8, 24$. \textcolor{black}{These norms are approximated from the solution vectors $u^{n}$ and $\varphi^{n}$ of the discrete system using a trapezoidal rule for the integrals.} The $L^2$ error as a function of time is shown in Figure \ref{fig:FSIthirdploterror}, corresponding to the numerical solutions depicted in Figure \ref{fig:FSIthirdplot}. All plots show excellent approximation of the simple behavior of the solution  for short times and a monotonous convergence on the whole time interval. 
Note that error \textcolor{black}{in Figure \ref{fig:FSIthirdploterror}} does not seem to grow with time, as expected for a variational method. 
Figure \ref{fig:FSIthirdplotconvDOF} considers the convergence of the numerical solutions  up to $n=24$ in terms of \textcolor{black}{the mesh size $h$. It depicts the $L^2(\Omega \times [0,T])$-norm of the error in ${\bf u}$, as well as the $L^2(\Gamma \times [0,T])$-norm of the error in $\varphi$. Similar convergence rates are obtained for ${\bf u}$ and $\varphi$ in these $L^2$-norms: $0.76$ for ${\bf u}$, $0.80$ for $\varphi$. Note that based on Theorem \ref{mainthm} and the trace theorem for Sobolev spaces, one might naively expect slower convergence  in $\varphi$ than in ${\bf u}$ (by a difference of the rates $0.5$) in the norms used here. However, as shown in \cite{mpw} for time-independent FEM-BEM coupling, under mild regularity assumptions the BEM solution $\varphi$  converges at a rate $0.5$ faster than predicted from a joint estimate for $({\bf u}, \varphi)$ as in Theorem \ref{mainthm}. This exactly cancels the above difference of rates and leads to identical convergences rates for  ${\bf u}$, $\varphi$ in the space-time $L^2$-norms on $\Omega$, resp.~$\Gamma$. The identical observed convergence rates for ${\bf u}$, $\varphi$ are therefore expected and in line with those known for FEM-BEM coupling in time-independent problems \cite{mpw}.}


\begin{figure}[htbp]
 \includegraphics[width=0.32\textwidth]{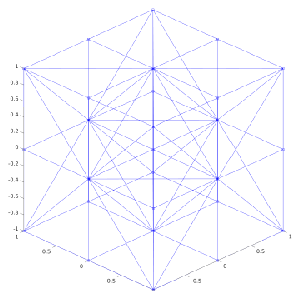}
 \includegraphics[width=0.32\textwidth]{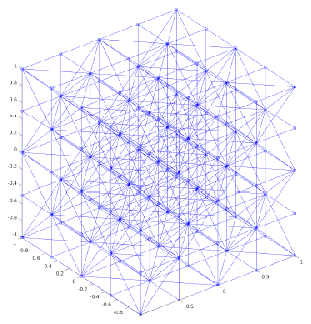}
 \includegraphics[width=0.32\textwidth]{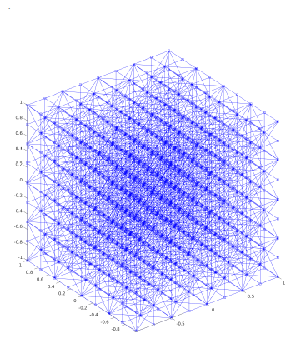}
 \caption{Meshes for $[-1,1]^3$ with 27 (n=2), 125 (n=4) and 729 (n=8) nodes.}\label{fig:meshplot}
\end{figure}
\begin{figure}[htbp]
 \includegraphics[width=\textwidth]{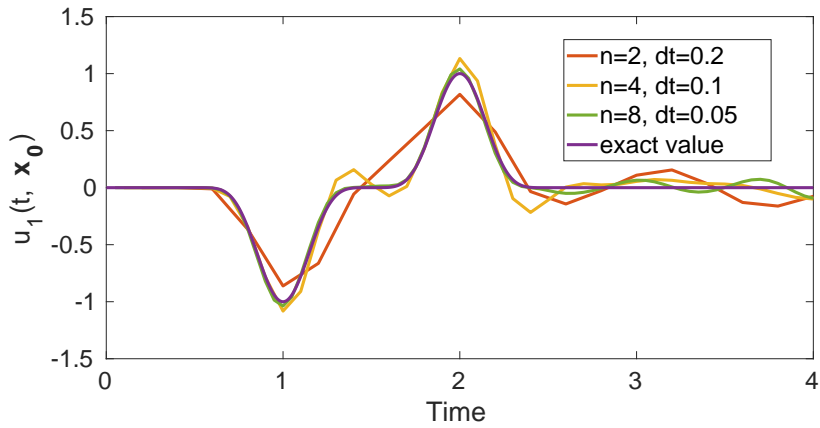}
 \caption{Numerical and exact solutions $u_1(t,\bf{x_0})$,  $\bf{x_0}=(-1,-1,-1)$.}\label{fig:FSIat-1-1-1}
\end{figure}
\begin{figure}[htbp]
 \includegraphics[width=\textwidth]{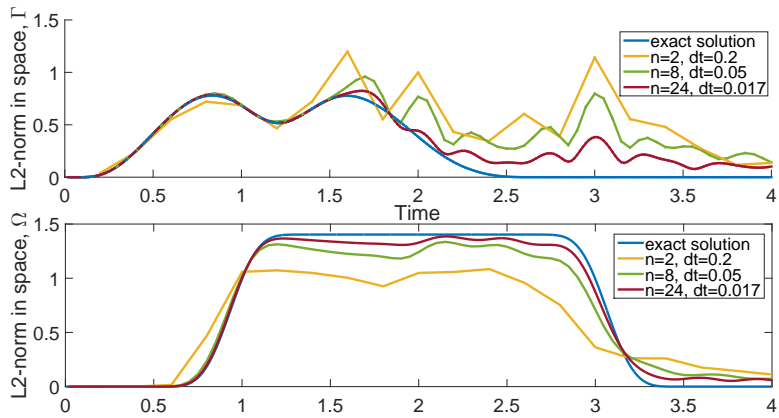}
 \caption{$L^2$-norm in space of the exact and numerical solutions for ${\bf u}$, resp.~$\varphi$.}\label{fig:FSIthirdplot}
\end{figure}

\begin{figure}[htbp]
 \includegraphics[width=\textwidth]{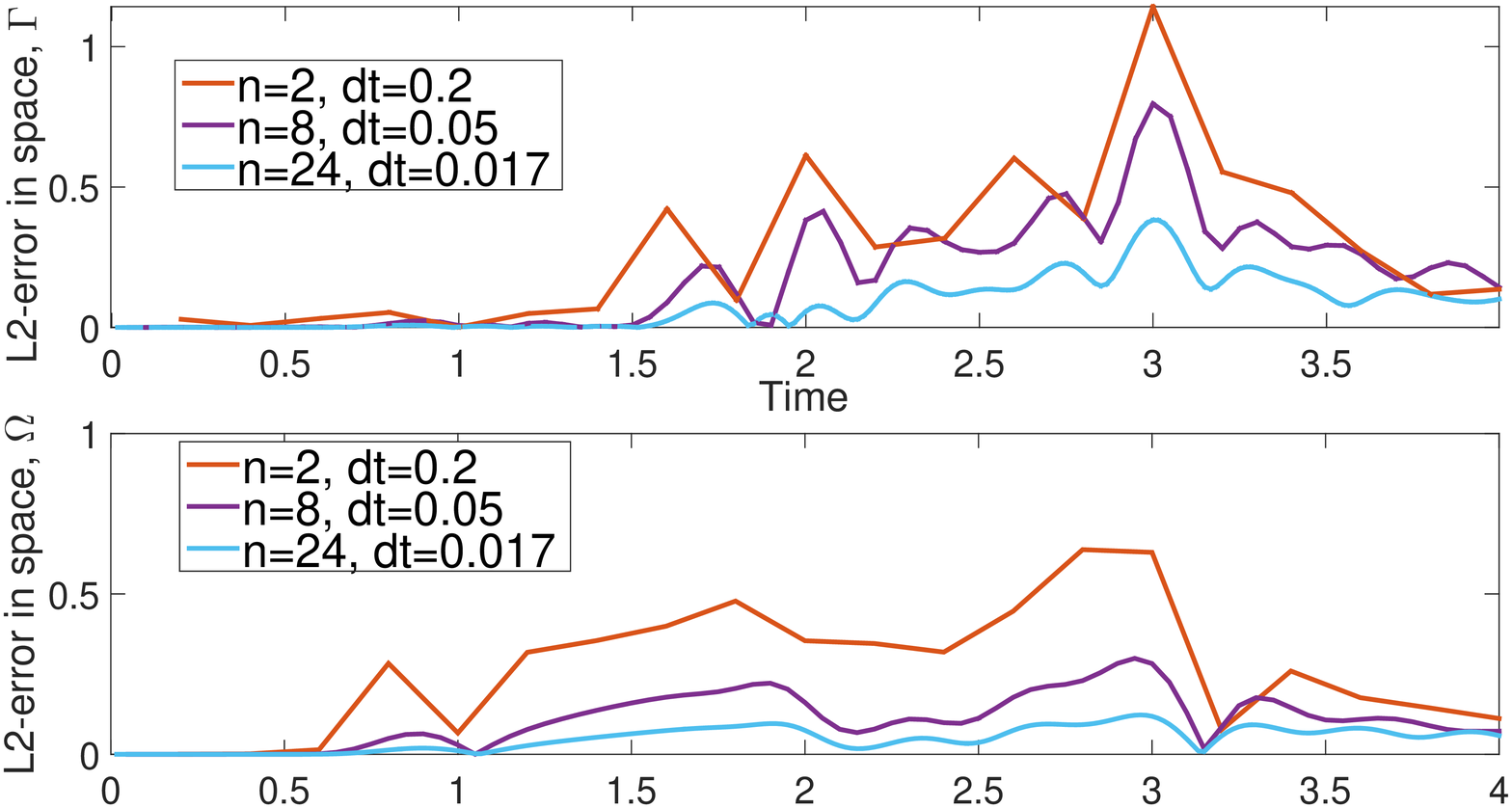}
 \caption{$L^2$-error in space for ${\bf u}$, resp.~$\varphi$.}\label{fig:FSIthirdploterror}
\end{figure}

\begin{figure}[htbp]
  \includegraphics[width=\textwidth]{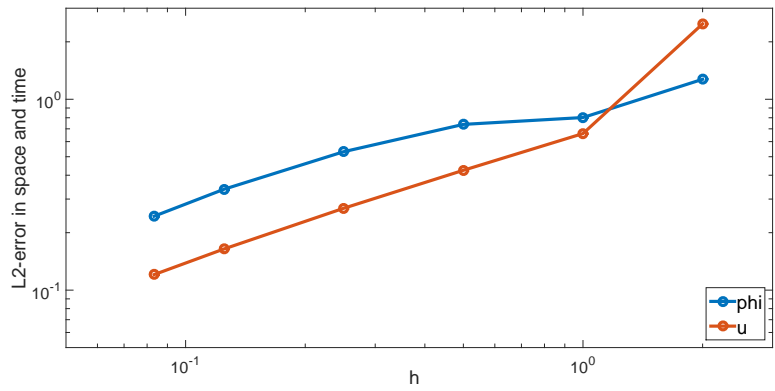}
\caption{$L^{2}$-error  in space-time for ${\bf u}$, $\varphi$ in terms of $h$.}\label{fig:FSIthirdplotconvDOF}
\end{figure}


\appendix
\section{Integral operators and finite--boundary elements}\label{SECT:retintop}

\textcolor{black}{We recall basic definitions and properties of boundary integral operators for the wave equation from \cite{MR3468871}, as well as from \cite{costabel04, MR2032873}.}

Let $\Gamma$ be the boundary of a polyhedral domain $\Omega$ in $\mathbb{R}^3$, consisting of curved, polygonal boundary faces. In $\mathbb{R}^3\backslash \Gamma$, a solution $v$ to the homogeneous wave equation may be represented in terms of the jump of the Dirichlet and Neumann data across $\Gamma$: $v = D \phi - S \lambda$. Here \textcolor{black}{for $x \in \mathbb{R}^{3}\backslash\Gamma$ and $t\geq 0$}
\begin{align}
S \lambda(x,t) &= \int_0^\infty\int_{\Gamma} G(t- \tau,x,y)\ \lambda(y,\tau)\ dy \ d\tau \ , \\
D \phi(x,t) &= \int_0^\infty\int_{\Gamma} \frac{\partial G}{\partial n_y}(t- \tau,x,y)\ \phi(y,\tau)\ dy \ d\tau\ ,
\end{align} are the single, resp.~double layer potential for the wave equation defined from the fundamental solution $G(\tau, x,y) = \frac{\delta(\tau - |x-y|)}{4\pi|x-y|}$.

The coupling method presented in this article relies on the resulting boundary integral operators \textcolor{black}{on $\Gamma \times (0,\infty)$}. \textcolor{black}{For $(x,t) \in \Gamma \times (0,\infty)$ we define }
\begin{align}\label{operators}
\nonumber V\phi(x,t)&=\int_0^\infty\int_{\Gamma} G(t- \tau,x,y)\ \phi(y,\tau)\ dy \ d\tau\ ,\\
\nonumber K\phi(x,t)&=\int_0^\infty\int_{\Gamma} \frac{\partial G}{\partial n_y}(t- \tau,x,y)\ \phi(y,\tau)\ dy \ d\tau\ ,\\
K^T \phi(x,t) = K' \phi(x,t)&= \int_0^\infty\int_{\Gamma} \frac{\partial G}{\partial n_x}(t- \tau,x,y)\ \phi(y,\tau)\ dy \ d\tau\, ,\\
\nonumber W \phi(x,t)&= \int_0^\infty\int_{\Gamma} \frac{\partial^2 G}{\partial n_x \partial n_y}(t- \tau,x,y)\ \phi(y,\tau)\ dy \ d\tau\ .
\end{align}
They are studied in space-time anisotropic Sobolev spaces $H_\sigma^{{r}}(\mathbb{R}^+,{H}^{{s}}(\Gamma))$ \cite{MR2032873}. 

To define an explicit scale of Sobolev norms, fix a partition of unity $\alpha_i$ subordinate to a covering of $\Gamma$ by open sets $B_i$ and diffeomorphisms $\phi_i$ mapping each $B_i$ into the unit cube $\subset \mathbb{R}^2$. They induce a family of norms from $\mathbb{R}^2$:
\begin{equation*}
 ||u||_{{{s}},{\Gamma}}=\left( \sum_{i=1}^p \int_{\mathbb{R}^2} (|\omega|^2+|\xi|^2)^{{s}}|\mathcal{F}\left\{(\alpha_i u)\circ \phi_i^{-1}\right\}(\xi)|^2 d\xi \right)^{\frac{1}{2}}\ .
\end{equation*}
$\mathcal{F}$ here denotes the Fourier transform. 
The norms for different $\omega \in \mathbb{C}\backslash \{0\}$ are equivalent.  

Weighted Sobolev spaces in time for $r\in\mathbb{R}$ and $\sigma>0$: are defined as
\begin{align*}
 H^{{r}}_\sigma(\mathbb{R}^+)&=\{ u \in \mathcal{D}^{'}_{+}: e^{-\sigma t} u \in \mathcal{S}^{'}_{+}  \textrm{ and }   \|u\|_{H^r_\sigma(\mathbb{R}^+)} < \infty \}\ .
\end{align*}
Here, $\mathcal{D}^{'}_{+}$ denotes the space of distributions on $\mathbb{R}$ with support in $[0,\infty)$, and  $\mathcal{S}^{'}_{+}$ the subspace of tempered distributions. The Sobolev spaces are Hilbert spaces endowed with the norm 
\begin{align*}
\|u\|_{H^r_\sigma(\mathbb{R}^+)}&=\left(\int_{-\infty+i\sigma}^{+\infty+i\sigma}|\omega|^{2r}\ |\hat{u}(\omega)|^2\ d\omega \right)^{\frac{1}{2}}\,.
\end{align*}
{The scale of space-time anisotropic Sobolev spaces \textcolor{black}{on $\Gamma$} combines the Sobolev norms in space and time:
\begin{definition}\label{sobdef}
For $r,s\in\mathbb{R}$ {and $\sigma>0$} define
\begin{align*}
 H^{{r}}_\sigma(\mathbb{R}^+,{H}^{{s}}(\Gamma))&=\{ u \in \mathcal{D}^{'}_{+}(H^{{s}}(\Gamma)): e^{-\sigma t} u \in \mathcal{S}^{'}_{+}(H^{{s}}(\Gamma))  \textrm{ and }   \|u\|_{{{r,s}},\Gamma} < \infty \}\ .
\end{align*}
$\mathcal{D}^{'}_{+}(E)$ denotes the space of distributions on $\mathbb{R}$ with support in $[0,\infty)$, taking values in $E = {H}^{{s}}({\Gamma})$, and  $\mathcal{S}^{'}_{+}(E)$ the subspace of tempered distributions. These Sobolev spaces are Hilbert spaces endowed with the norm 
\begin{align*}
\|u\|_{{{r,s}},\Gamma}&=\left(\int_{-\infty+i\sigma}^{+\infty+i\sigma}|\omega|^{2{{r}}}\ \|\hat{u}(\omega)\|^2_{{{s}},\Gamma}\ d\omega \right)^{\frac{1}{2}}\ .
\end{align*}
\end{definition}
When $|s|\leq 1$ one can show that the spaces are independent of the choice of $\alpha_i$ and $\phi_i$. 

\textcolor{black}{In a bounded Lipschitz domain $\Omega \subset \mathbb{R}^d$,} we define space-time anisotropic Sobolev spaces  \textcolor{black}{$H^{{r}}_\sigma(\mathbb{R}^+,{H}^{{s}}(\Omega))$ analogously to above}, starting from   the standard Sobolev spaces $H^s(\Omega)$ with norm $\|u\|_{s,\Omega} = \inf_{v} \left(\int_{\mathbb{R}^d} (|\omega|^2+|\xi|^2)^{{s}}|\mathcal{F}v(\xi)|^2 d\xi \right)^{\frac{1}{2}}$. Here the infimum extends over \textcolor{black}{all extensions $v \in H^s(\mathbb{R}^d)$ of $u \in H^s(\Omega)$, i.e.~all $v$ with $v|_\Omega = u$}.\\

We state the mapping properties of the boundary integral operators, see e.g.~\cite{costabel04, MR2032873}:
\begin{theorem}\label{mappingproperties}
The following operators are continuous for $r\in \mathbb{R}$, ${{\sigma>0}}$:
\begin{align*}
& V:  {H}^{r+1}_\sigma(\mathbb{R}^+, {H}^{-\frac{1}{2}}(\Gamma))\to  {H}^{r}_\sigma(\mathbb{R}^+, {H}^{\frac{1}{2}}(\Gamma)) \ ,
\\ & K':  {H}^{r+1}_\sigma(\mathbb{R}^+, {H}^{-\frac{1}{2}}(\Gamma))\to {H}^{r}_\sigma(\mathbb{R}^+, {H}^{-\frac{1}{2}}(\Gamma)) \ ,
\\ & K:  {H}^{r+1}_\sigma(\mathbb{R}^+, {H}^{\frac{1}{2}}(\Gamma))\to {H}^{r}_\sigma(\mathbb{R}^+, {H}^{\frac{1}{2}}(\Gamma)) \ ,
\\ & W:  {H}^{r+1}_\sigma(\mathbb{R}^+, {H}^{\frac{1}{2}}(\Gamma)))\to {H}^{r}_\sigma(\mathbb{R}^+, {H}^{-\frac{1}{2}}(\Gamma)) \ .
\end{align*}
\end{theorem}

\textcolor{black}{By a fundamental observation of Bamberger and Ha-Duong \cite{BamHa},} $V\partial_t$ satisfies a coercivity estimate in the norm of ${H}^{0}_\sigma(\mathbb{R}^+, {H}^{-\frac{1}{2}}(\Gamma))$, \textcolor{black}{provided $\sigma >0$}:  $\|\psi\|^2_{0,-\frac{1}{2}, \Gamma} \lesssim_\sigma \langle V \psi,  \dot{\psi}\rangle$. From the mapping properties of Theorem \ref{mappingproperties} one also has  the continuity of the bilinear form associated to $V \partial_t$ in a bigger norm: $\langle V \psi, \dot{\psi}\rangle \lesssim \|\psi\|^2_{1,-\frac{1}{2}, \Gamma}$.  Similar estimates hold for $W\partial_t$: $\|\phi\|^2_{0,\frac{1}{2}, \Gamma} \lesssim_\sigma \langle W \phi, \dot{\phi}\rangle\lesssim \|\phi\|^2_{1,\frac{1}{2}, \Gamma}$. Proofs and further information may be found in  \cite{MR2032873}.\\

\textcolor{black}{For sufficiently regular $\phi$ and $\lambda$, the following jump relations hold \cite{MR2032873}:
\begin{subequations}
 \begin{equation*}
    (S \lambda)^{-}|_{\Gamma} = (S \lambda)^{+}|_{\Gamma} = V \lambda ,  \quad \partial_n^{-}(D \phi)=\partial_n^{+}(D \phi) = W \phi , 
 \end{equation*}
   \begin{equation*}
    \partial_n^{+}(S \lambda) = (- \frac{1}{2}I + K') \lambda ,  \quad
    \partial_n^{-}(S \lambda) = (\frac{1}{2}I + K') \lambda , 
 \end{equation*}
   \begin{equation*}
    (D \phi)^{+}|_{\Gamma} = (\frac{1}{2}I + K) \phi , 
  \quad
    (D \phi)^{-}|_{\Gamma} = (-\frac{1}{2}I + K) \phi .
 \end{equation*}
\end{subequations}
}

We consider space-time discretizations based on tensor products of piecewise polynomials:

For simplicity, we assume that $\Omega$ is a polygonal domain, with a quasi-uniform triangulation $\mathcal{T}_\Omega ={\{T_1,\cdots, T_{N_o}\}}$by $\textcolor{black}{N_{o}}$ tetrahedra. The induced quasi-uniform triangulation of the  boundary $\Gamma$, $\mathcal{T}_S={\{\Delta_1,\cdots,\Delta_{N_{s'}}\}}$, should consist of closed triangular faces $\Delta_i$, such that each $\Delta_i$ is a face of one $T_j$ and at most one face of $T_j$ is contained in $\Gamma$. 

We consider the space $V_h^q(\Omega)$ of piecewise polynomial functions \textcolor{black}{on $\mathcal{T}_\Omega$} of degree $q\geq 0$ in
space \textcolor{black}{(continuous if $q\geq 1$)}. $V_h^q$ consists of traces on $\Gamma$ of functions in $V_h^q(\Omega)$.  \textcolor{black}{The parameter $h$ denotes the maximal diameter of an element in $\mathcal{T}_\Omega$.}

We choose an equidistant temporal mesh on the positive half-line $\mathcal{T}_T=\{[0,t_1),[t_1,t_2),\dots\}$, where $t_n = n (\Delta t)$. $V^p_{\Delta t}$ is the space of piecewise  polynomial  functions of degree $p$ on $\mathcal{T}_T$ (continuous and vanishing at $t=0$ if $p\geq 1$, $C^1$ if $p\geq 2$).

The space-time approximation spaces are given by tensor products of the approximation spaces in space and time, $V_h^q$ and $V^p_{\Delta t}$, associated to the space-time meshes $\mathcal{T}_{\Omega,T}=\mathcal{T}_\Omega \times\mathcal{T}_T$, respectively $\mathcal{T}_{S,T}=\mathcal{T}_S \times\mathcal{T}_T$. We write
\begin{align}\label{eq:discretspace}
V_{\Delta t,h}^{p,q}(\Omega):=  V_{\Delta t}^p \otimes V_{h}^q(\Omega),\ \ V_{\Delta t,h}^{p,q}:=  V_{\Delta t}^p \otimes V_{h}^q\ .
\end{align}

\section{Discretization and MOT-Algorithm}

This appendix discusses the details of the discretization \eqref{eq:FSIh}, 
where we set $\sigma=0$. 
\textcolor{black}{The resulting formulas for the entries of the Galerkin matrices reduce their assembly to numerical quadratures over certain light cones $E_l$ below. This structure is crucial for the practical implementation  in standard time-domain boundary element codes, see \cite{gimperleinreview, terrasse93}, as well as the recent Ph.D.~thesis \cite{phdoezdemir} of the second author.}
 
We choose the finite element ansatz and test functions as in \eqref{eq:discinnansatz} - \eqref{eq:discboundVtest}.

For the discretization of the boundary integral operators, we begin with the retarded hypersingular operator. 

We choose the ansatz function as in \eqref{eq:discboundWansatz} and the test function as in \eqref{eq:discboundWtest}. 
\begin{flalign*}
 &\langle W \phi_{h,\triangle t} , \dot{w}_{h,\Delta t} \rangle_{\Gamma \times \mathbb{R}_{+}} = \sum\limits_{m=1}^{N_t} \sum\limits_{i=1}^{N_{s'}} \varphi^{m}_{i} \Big[ - \iint\limits_{E_{n-m}} \frac{(n_x \cdot n_y) \xi_{h}^{i}(y) \xi_{h}^{j}(x)}{(\triangle t) |x-y| 4 \pi} ds_y ds_x \\ &+ 2 \iint\limits_{E_{n-m-1}} \frac{(n_x \cdot n_y) \xi_{h}^{i}(y) \xi_{h}^{j}(x)}{(\triangle t) |x-y| 4 \pi} ds_y ds_x -   \iint\limits_{E_{n-m-2}} \frac{(n_x \cdot n_y) \xi_{h}^{i}(y) \xi_{h}^{j}(x)}{(\triangle t) |x-y| 4 \pi} ds_y ds_x \Big] \\
 &+ \sum\limits_{m=1}^{N_t} \sum\limits_{i=1}^{N_{\tilde{s}}} \varphi^{m}_{i} \iint_{\Gamma \times \Gamma} \frac{(\mathop{curl}_{\Gamma} \xi_{h}^{i})(y)\cdot (\mathop{curl}_{\Gamma} \xi_{h}^{j})(x))}{4 \pi |x-y|} \mathcal{Y}^{n-m} ds_y ds_x \\ &=: \sum_{m=1}^{N_t} \sum_{i=1}^{N_{s'}} W^{n-m}_{j,i} \varphi_{i}^{m} 
:= \sum_{m=1}^{N_t} W^{n-m} \varphi^{m}\ ,
\end{flalign*}
where \cite{phdoezdemir}
\begin{align*}
 & \mathcal{Y}^{n-m}(x,y) = (2 (\triangle t))^{-1} (|x-y|^{2}-2 |x-y| (n-m+1) (\triangle t) + ((n-m+1) (\triangle t))^{2}) \chi_{E_{n-m}} \\ &+ (2 (\triangle t))^{-1} (|x-y|^{2}-2 |x-y| (n-m-2) (\triangle t) + ((n-m-2) (\triangle t))^{2}) \chi_{E_{n-m-2}} \\ &+ (2 (\triangle t))^{-1} (-2 |x-y|^{2}+2 |x-y| ( (n-m-1) (\triangle t) + (n-m) (\triangle t) ) \\ & - ( ( (n-m-1) (\triangle t))^{2} + ((n-m) (\triangle t))^{2}  ) + 2 (\triangle t)^{2} ) \chi_{E_{n-m-1}} \ .
\end{align*}
Here, for $l \in \mathbb{N}_{0}$ we define the light cone $E_l = \lbrace (x,y) \in \Gamma \times \Gamma : t_{l} \leq |x-y| \leq t_{l+1} \rbrace \subset \Gamma \times \Gamma$, and $\chi_{E_{l}}(x,y)=1$ if $(x,y)\in E_l$, and $=0$ otherwise. \\

The matrix $W^{n-m}$ is therefore a sum of integrals over the \textcolor{black}{three} light cones $E_{n-m}, E_{n-m-1}$ and $E_{n-m-2}$. 

We now consider the discretization of the single layer potential. For the ansatz function we choose \eqref{eq:discboundVansatz} 
and as test function we choose \eqref{eq:discboundVtest}. 

After some computations, we obtain 
\begin{flalign*}
& \langle V \lambda_{h,\triangle t}, \dot{m}_{h,\Delta t} \rangle_{\Gamma \times \mathbb{R}_{+}} = \int_{0}^{\infty} \int_{\Gamma} V \lambda_{h,\triangle t}(x,t) \cdot  \dot{m}_{h,\Delta t}(x,t) ds_x dt \\
&= \!\!\sum\limits_{m=1}^{N_t} \sum\limits_{i=1}^{N_{s'}} \lambda_{i}^{m} \Big[  \iint\limits_{E_{n-m}} \left(- (n-m+1) \frac{\xi_{h}^{i}(y) \xi_{h}^{j}(x)}{4 \pi |x-y|} + \frac{\xi_{h}^{i}(y) \xi_{h}^{j}(x)}{4 \pi (\triangle t)} \right) ds_y ds_x \\  &\!\!+\!\! \iint\limits_{E_{n-m-1}} \left(\!\!  (2 (n-m)-1) \frac{\xi_{h}^{i}(y) \xi_{h}^{j}(x)}{4 \pi |x-y|} - 2 \frac{\xi_{h}^{i}(y) \xi_{h}^{j}(x)}{4 \pi (\triangle t)} \right) ds_y ds_x \\ &\!\!+\!\! 
\iint\limits_{E_{n-m-2}} \left(\!\! - (n-m-2) \frac{\xi_{h}^{i}(y) \xi_{h}^{j}(x)}{4 \pi |x-y|} + \frac{\xi_{h}^{i}(y) \xi_{h}^{j}(x)}{4 \pi (\triangle t)} \right) ds_y ds_x \Big]\\
& =: \!\!\!\sum_{m=1}^{N_t} \sum_{i=1}^{N_{s'}} V_{j,i}^{n-m} \lambda_{i}^{m} \!=:\!\!\! \sum_{m=1}^{N_t} \!\! V^{n-m} \lambda^{m}. &
\end{flalign*}
We next consider the retarded adjoint double layer potential:
\begin{flalign*}
 &\langle K^{T} \lambda_{h,\triangle t} , \dot{w}_{h,\Delta t} \rangle_{\Gamma \times \mathbb{R}_{+}} = \int_{0}^{\infty} \int_{\Gamma} K^{T} \lambda_{h,\triangle t} \dot{w}_{h,\Delta t} ds_x dt &\\
 &= \sum\limits_{m=1}^{N_t} \sum\limits_{i=1}^{N_{s'}} \lambda_{i}^{m} \iint\limits_{\Gamma \times \Gamma} \frac{n_x \cdot (x-y)}{4 \pi |x-y|^{3}} \xi_{h}^{i}(y) \xi_{h}^{j}(x) \mathcal{Y}^{n-m}(x,y) ds_y ds_x &\\
 &+ \sum\limits_{m=1}^{N_t} \sum\limits_{i=1}^{N_{s'}} \lambda_{i}^{m} \Big[  \iint\limits_{E_{n-m}} n_x \cdot (x-y) \left( (n-m+1) \frac{\xi_{h}^{i}(y) \xi_{h}^{j}(x)}{4 \pi |x-y|^{2}} - \frac{\xi_{h}^{i}(y) \xi_{h}^{j}(x)}{4 \pi (\triangle t) |x-y|} \right) ds_y ds_x & \\  &+ \iint\limits_{E_{n-m-1}} n_x \cdot (x-y) \left( - (2 (n-m)-1) \frac{\xi_{h}^{i}(y) \xi_{h}^{j}(x)}{4 \pi |x-y|^{2}} + 2 \frac{\xi_{h}^{i}(y) \xi_{h}^{j}(x)}{4 \pi (\triangle t) |x-y|} \right) ds_y ds_x & \\ &+ 
\iint\limits_{E_{n-m-2}} n_x \cdot (x-y) \left( (n-m-2) \frac{\xi_{h}^{i}(y) \xi_{h}^{j}(x)}{4 \pi |x-y|^{2}} - \frac{\xi_{h}^{i}(y) \xi_{h}^{j}(x)}{4 \pi (\triangle t) |x-y|} \right) ds_y ds_x \Big] &
\\ &= \sum_{m=1}^{N_t} \sum_{i=1}^{N_{s'}} ({K^{T}})^{n-m}_{j,i} \lambda_{i}^{m} = \sum_{m=1}^{N_t} {K^{T}}^{n-m} \lambda^{m} \ .&
\end{flalign*}
The related term for the mass matrix is given by
\begin{flalign*}
&{\textstyle \langle \frac{1}{2}}  \lambda_{h,\triangle t} , \dot{w}_{h,\Delta t} \rangle_{\Gamma \times \mathbb{R}_{+}} 
= \frac{1}{2} \int\limits_{0}^{\infty} \int_{\Gamma} \sum\limits_{m=1}^{N_t} \sum\limits_{i=1}^{N_{s'}} \lambda_{i}^{m} \beta_{\Delta t}^{m}(t) \xi_{h}^{i}(x) {\gamma}_{\Delta t}^{n}(t) \xi_{h}^{j}(x) ds_x dt &\\
&= \frac{1}{2} \sum_{m=1}^{N_t} \sum_{i=1}^{N_{s'}} \lambda_{i}^{m} (\int_{\Gamma} \xi_{h}^{i}(x) \xi_{h}^{j}(x) ds_x) (\int_{0}^{\infty} \beta_{\Delta t}^{m}(t) {\gamma}_{\Delta t}^{n} dt) &\\ &= \frac{1}{2}  \sum_{i=1}^{N_{s'}} \lambda_{i}^{m} (\int_{\Gamma} \xi_{h}^{i}(x) \xi_{h}^{j}(x) ds_x) \frac{(\triangle t)}{2} \begin{cases} \lambda_{i}^{1} & , n=1 \\ \lambda_{i}^{n}+\lambda_{i}^{n-1} & , n \geq 2 \end{cases} &\\
&=: \frac{1}{2} \sum_{i=1}^{N_{s'}} I_{j,i} \lambda_{I}
= \frac{1}{2} I \lambda_I  \ ,&
\end{flalign*}
where 
\begin{equation*}
  \lambda_{I}= \frac{(\triangle t)}{2} \begin{cases} \lambda^{1} & , n=1 \\
	  \lambda^{n} + \lambda^{n-1} & , n \geq 2 \ \ . \end{cases} 
\end{equation*}
Furthermore 
\begin{flalign*}
 &\langle K \phi_{h,\triangle t} , \dot{m} \rangle_{\Gamma \times \mathbb{R}_{+}} = \int_{0}^{\infty} \int_{\Gamma} K \phi_{h,\triangle t} \dot{m}_{h,\Delta t} ds_x dt \\
 &= \sum\limits_{m=1}^{N_t} \sum\limits_{i=1}^{N_{s'}} \varphi_{i}^{m} 
 \Big[  \!\!\iint\limits_{E_{n-m}} \!\!n_y \!\cdot\! (x-y) \left( - (n-m+1) \frac{\xi_{h}^{i}(y) \xi_{h}^{j}(x)}{4 \pi |x-y|^{3}} \!+\! \frac{\xi_{h}^{i}(y) \xi_{h}^{j}(x)}{4 \pi (\triangle t) |x-y|^{2}} \right) ds_y ds_x \\ &+ \iint\limits_{E_{n-m-1}} n_y \cdot (x-y) \left( (2 (n-m)-1) \frac{\xi_{h}^{i}(y) \xi_{h}^{j}(x)}{4 \pi |x-y|^{3}} - 2 \frac{\xi_{h}^{i}(y) \xi_{h}^{j}(x)}{4 \pi (\triangle t) |x-y|^{2}} \right) ds_y ds_x \\ &+ 
\iint\limits_{E_{n-m-2}} n_y \cdot (x-y) \left( - (n-m-2) \frac{\xi_{h}^{i}(y) \xi_{h}^{j}(x)}{4 \pi |x-y|^{3}} + \frac{\xi_{h}^{i}(y) \xi_{h}^{j}(x)}{4 \pi (\triangle t) |x-y|^{2}} \right) ds_y ds_x \Big] \\  &+ \!\!\sum\limits_{m=1}^{N_t} \!\sum\limits_{i=1}^{N_{s'}} \!\!\varphi_{i}^{m} 
\Big[ \!\!\iint\limits_{E_{n-m}} \!\!\!\! \frac{-n_y \!\cdot\! (x-y)}{4 \pi (\!\triangle t\!) |x\!-\!y|^{2}} \xi_{h}^{i}(\!y\!) \xi_{h}^{j}(\!x\!) ds_y ds_x \!\!+\!\!\!\!\!\!\!\iint\limits_{E_{n-m-1}} \!\!\!\!\!\!\frac{2 n_y \!\cdot\!(x-y)}{4 \pi (\!\triangle t\!) |x\!-\!y|^{2}} \xi_{h}^{i}(\!y\!) \xi_{h}^{j}(\!x\!) ds_y ds_x & \\ &+ \iint\limits_{E_{n-m-2}} - \frac{n_y \cdot (x-y)}{4 \pi (\!\triangle t\!) |x-y|^{2}} \xi_{h}^{i}(y) \xi_{h}^{j}(x) ds_y ds_x  \Big]
= \sum_{m=1}^{N_t} \sum_{i=1}^{N_{s'}} {K}^{n-m}_{j,i} \varphi_{i}^{m} = \sum_{m=1}^{N_t} {K}^{n-m} \varphi^{m}
\end{flalign*}
and
\begin{flalign*}
&\langle {\textstyle \frac{1}{2}} \phi_{h,\triangle t} , \dot{m}_{h,\Delta t} \rangle_{\Gamma \times \mathbb{R}_{+}} 
= \frac{1}{2} \int_{0}^{\infty} \int_{\Gamma} \sum\limits_{m=1}^{N_t} \sum\limits_{i=1}^{N_{s'}} \varphi_{i}^{m} \beta_{\Delta t}^{m}(t) \xi_{h}^{i}(x) \dot{\gamma}_{\Delta t}^{n}(t) \xi_{h}^{j}(x) ds_x dt &\\
&= \frac{1}{2} \sum_{i=1}^{N_{s'}} \varphi_{i}^{m} (\int_{\Gamma} \xi_{h}^{i}(x) \xi_{h}^{j}(x) ds_x) (\int_{0}^{\infty} \beta_{\Delta t}^{m}(t) \dot{\gamma}_{\Delta t}^{n} dt) &\\ &= \frac{1}{2} \sum_{m=1}^{N_t} \sum_{i=1}^{N_{s'}} \varphi_{i}^{m} (\int_{\Gamma} \xi_{h}^{i}(x) \xi_{h}^{j}(x) ds_x) \begin{cases} - \varphi_{i}^{1} & , n=1 \\ - (\varphi_{i}^{n}-\varphi_{i}^{n-1}) & , n \geq 2 \end{cases} &\\
&= \frac{1}{2} \sum_{i=1}^{N_{s'}} I_{j,i} \varphi_{I}
= \frac{1}{2} I \varphi_{I}  \ ,
\end{flalign*}
with 
\begin{equation*}
  \varphi_{I}= \begin{cases} - \varphi^{1} & , n=1 \\
	  - \varphi^{n} + \varphi^{n-1} & , n \geq 2 \ \ . \end{cases}
\end{equation*}It remains to consider the coupling contributions. For $j=1,\ldots,N_{s'}$ 
\begin{flalign*}
 &\langle \dot{\mathbf{{u}}}_{h,\triangle t}\textcolor{black}{|_{\Gamma}} \cdot n ,  \dot{w}_{h,\Delta t} \rangle_{\Gamma \times \mathbb{R}_{+}} =  \int_{0}^{\infty} \int_{\Gamma} 
  \dot{\mathbf{{u}}}_{h,\triangle t}\textcolor{black}{|_{\Gamma}} \cdot n_{\textcolor{black}{x}} \dot{w}_{h,\Delta t} ds_x dt  & \\ &= 
  \sum_{\nu=1}^{3} \sum_{i=1}^{N_o} (\int_{\Gamma} \eta_{h}^{i}|_{\Gamma}(x) \vec{e}_{\nu} \cdot n_{\textcolor{black}{x}} \ \xi_{h}^{j}(x) ds_x ) \begin{cases} u_{\nu,i}^{1} & , n=1 \\ u_{\nu,i}^{n}-u_{\nu,i}^{n-1} & , n \geq 2   \end{cases} & \\
&=: \sum_{\nu=1}^{3} \sum_{i=1}^{N_o} (RI\!n_{x})_{(i,\nu),j} u_{T}
= RI\!n_{x} u_{T} \ ,&
\end{flalign*}
with 
\begin{equation*}
 u_{T} = \begin{cases} u^{1}_{\Gamma} & , n=1 \\ 
  u^{n}_{\Gamma}-u^{n-1}_{\Gamma} & , n \geq 2 \ \ .\end{cases}
\end{equation*}
For the second coupling term:
\begin{flalign*}
 \langle \dot{\phi}_{h,\triangle t} \cdot n , \dot{\mathbf{w}}_{h,\Delta t}|_\Gamma \rangle_{\Gamma \times \mathbb{R}_{+}} &=  \int_{0}^{\infty} \int_{\Gamma} \dot{\phi}_{h,\triangle t} \cdot n_{\textcolor{black}{x}} \dot{\mathbf{w}}_{h,\Delta t}|_\Gamma ds_x dt \\
 &= \sum_{i=1}^{N_{s'}} (\int_{\Gamma} \xi_{h}^{i}(x) n_{\textcolor{black}{x}} \cdot \eta_{h}^{j}|_{\Gamma}(x) \vec{e}_{\mu} ds_x) \begin{cases} \varphi_{i}^{1} & , n=1 \\ \varphi_{i}^{n}-\varphi_{i}^{n-1} & , n \geq 2  \end{cases} \\
&=: \sum_{i=1}^{N_{s'}} (n_{x}\!RI)_{i,(j,\eta)} \varphi_{T}
= n_{x}\!RI \varphi_{T}
\end{flalign*}
with 
\begin{equation*}
\varphi_{T} = \begin{cases} \varphi^{1} & , n=1 \\ 
  \varphi^{n}-\varphi^{n-1} & , n \geq 2 \ \ . \end{cases}
\end{equation*}
For completeness we mention the right hand side: Set $h=\dot{v}^{inc} n $ and \textcolor{black}{$g={\partial_{n}^{+}v^{inc}}$}. 
We approximate the time integral by the trapezoidal rule, so that:
\begin{flalign*}
&- \langle \dot{v}^{inc}n, {\dot{\mathbf{{w}}}|_\Gamma}\rangle_{\Gamma \times \mathbb{R}_{+}} = - \langle h , {\dot{\mathbf{{w}}}|_\Gamma} \rangle_{\Gamma \times \mathbb{R}_{+}}
= - \frac{(\triangle t)}{2} \int_{\Gamma} (h^{n}\!+\!h^{n-1}) \eta_{h}^{j}(x) e_{\mu} ds_x 
=: H^{n} + H^{n-1} \\
& \langle \textcolor{black}{\textstyle {\partial_n^{+}}v^{inc}}, \dot{w}\rangle_{\Gamma} = \langle g , \dot{w} \rangle_{\Gamma}
= \frac{(\triangle t)}{2} \int_{\Gamma} (g^{n}+g^{n-1}) \xi_{h}^{j}(x) ds_x 
=: G^{n} + G^{n-1}\ ,
\end{flalign*}
where $h^{n}=h(x,t_{n})$ and $g^{n}=g(x,t_{n})$.\\

\textcolor{black}{Defining 
 $$u_{A} = (\triangle t) \cdot \begin{cases} \frac{1}{2} u^{1} & , n=1 \\ \frac{u^{n}+u^{n-1}}{2} & , n \geq 2 \end{cases} \ \  \text{ and  }\ \ u_M=\frac{1}{(\triangle t)} \cdot \begin{cases}  u^{1} & , n=1 \\ u^{2}-2 u^{1} & , n=2 \\ u^{n}-2 u^{n-1}+u^{n-2} & , n\geq 3\end{cases} \ ,$$ }
the resulting system of equations therefore becomes 
\begin{align}
&A u_A + M u_M -RI\!n_{x} u_T + n_{x}\!RI \varphi_{T} - \sum_{m=1}^{N_t} W^{n-m} \varphi^{m} + \sum_{m=1}^{N_t} {K^{T}}^{n-m} \lambda^{m} 
-\frac{1}{2} I \lambda_{I} \nonumber \\ &+ \frac{1}{2} I \varphi_{I} - \sum_{m=1}^{N_t} K^{n-m} \varphi^{m} + \sum_{m=1}^{N_t} V^{n-m} \lambda^{m} = H^{n}+H^{n-1}+G^{n}+G^{n-1} \ . \label{eqtosolve}
\end{align}
The matrices $W^{k},K^{k},{K^{T}}^{k},V^{k}$ vanish if the index $k$ is negative. Therefore we get for \eqref{eqtosolve}
in the first time step $(n=1)$:
\begin{align*}
 \begin{pmatrix}
  \frac{(\triangle t)}{2} A + \frac{1}{(\triangle t)} M & [0, n_{x}\!RI]^{T} & 0 \\ [0,-RI\!n_{x}] & -W^{0} & {K^{T}}^{0}-\frac{1}{2} \frac{(\triangle t)}{2} I \\ 0 & -K^{0} - \frac{1}{2} I & V^{0} \end{pmatrix}
 \begin{pmatrix} u^{1} \\ \varphi^{1} \\ \lambda^{1} \end{pmatrix}
 = \begin{pmatrix} H^{1} + H^{0} \\ G^{1} +G^{0} \\ 0 \end{pmatrix}\ .
\end{align*}
\textcolor{black}{Note the zero block in $[0, -RIn_{x}]$, as $\mathbf{u}_{h,\Delta t}|_\Gamma$ only depends on the values in the nodes  on the boundary $\Gamma$. Similarly, one obtains a zero block in $[0, n_{x}RI]^{T}$, corresponding to the vanishing contribution of  nodes in the interior of $\Omega$ to the trace of the test function $\mathbf{w}_{h,\Delta t}|_\Gamma$.} 

For the second time step $(n=2)$ we obtain
\begin{align*}
 \!\begin{pmatrix}
  \!\frac{(\triangle t)}{2} A \!+\!\! \frac{1}{(\triangle t)} M & [0,n_{x}\!RI]^{T} & 0 \\ [0,-RI\!n_{x}] & -W^{0} & {K^{T}}^{0}\!\!\!-\!\! \frac{(\!\triangle t\!)}{4} I \\ 0 & -K^{0} \!-\! \frac{1}{2} I \!\!\! & V^{0} \end{pmatrix} \!
 \begin{pmatrix} \!\!u^{2} \! \\ \!\!\varphi^{2} \! \\ \!\!\lambda^{2} \! \end{pmatrix} \!
 =\! \begin{pmatrix} \!\!H^{2} \!+\! H^{1} \!\!-\! A \frac{(\triangle t)}{2} u^{1} \!+\! M \frac{2}{(\triangle t)} u^{1} \!+\! n_{x}\!RI \varphi^{1}  \\ \!\!G^{2} \!\!+\!G^{1}\!\!+\!\!RI\!n_{x} u^{1}_{\Gamma} \!+\! W \!\varphi^{1} \!\!+\!{K^{T}}^{1} \!\lambda^{1} \!\!+\!\!\frac{(\triangle t)}{4} I \lambda^{\!1} \!\! \\ K^{1} \varphi^{1} \!-\! \frac{1}{2} I \varphi^{1} - V^{1} \lambda^{1} \end{pmatrix}\!,
\end{align*}
using $u^{1},\varphi^{1}$ and $\lambda^{1}$ from above.  For later time steps  $n \geq 3$ we conclude:
\begin{align*}
 &\begin{pmatrix}
  \frac{(\triangle t)}{2} A + \frac{1}{(\triangle t)} M & [0,n_{x}\!RI]^{T} & 0 \\ [0,-RI\!n_{x}] & -W^{0} & {K^{T}}^{0}-\frac{1}{2} \frac{(\triangle t)}{2} I \\ 0 & -K^{0} - \frac{1}{2} I & V^{0} \end{pmatrix}
 \begin{pmatrix} u^{n} \\ \varphi^{n} \\ \lambda^{n} \end{pmatrix} \\
 &= \begin{pmatrix} H^{n} + H^{n-1} - A \frac{(\triangle t)}{2} u^{n-1} + M \frac{2}{(\triangle t)} u^{n-1} - M \frac{1}{(\triangle t)} u^{n-2} +n_{x}\!RI \varphi^{n-1}\\ G^{n} +G^{n-1} + RI\!n_{x} u^{n-1}_{\Gamma}   + \sum_{m=1}^{n-1} W^{n-m} \varphi^{m} - \sum_{m=1}^{n-1} {K^{T}}^{n-m} \lambda^{m} + \frac{1}{2} \frac{(\triangle t)}{2} I \lambda^{n-1} \\  \sum_{m=1}^{n-1} K^{n-m} \varphi^{m} - \frac{1}{2} I \varphi^{n-1}- \sum_{m=1}^{n-1} V^{n-m} \lambda^{m} \end{pmatrix}.
\end{align*}
This system is solved repeatedly until reaching time step $N_t \geq 3$.

\end{document}